\documentclass[12pt,letterpaper]{amsart}
\usepackage{amsthm,amsfonts,amssymb,amsmath}
\usepackage{amsmath, bm}
\usepackage{tabularx}
\usepackage{graphicx}
\usepackage{tikz}
\usetikzlibrary{patterns}
\usetikzlibrary{arrows.meta}
\usetikzlibrary{arrows}
\usepackage{amscd}
\usepackage{parskip}
\usepackage{enumitem}
\usepackage{comment}  
\usepackage{cancel} 
\usepackage{mathtools} 
\usepackage{cleveref} 
\newcommand\T{\rule{0pt}{2.7ex}}       

\newcommand{\tcblue}{\textcolor{blue}}

\newcommand{\svw}[1]{{\color{black}#1}}

\newcommand{\bK}{\mathbb{K}}
 
\newcommand{\bQ}{\mathbb{Q}}

\newtheorem{theorem}{Theorem}[section]
\newtheorem{lemma}[theorem]{Lemma}
\newtheorem{corollary}[theorem]{Corollary}

\newtheorem{proposition}[theorem]{Proposition}

\theoremstyle{definition}
\newtheorem{definition}[theorem]{Definition}
\newtheorem{example}[theorem]{Example}

\newtheorem{remark}[theorem]{Remark}

\newcommand{\poRI}{\preccurlyeq_{\mathcal{R}{\mathfrak{S}}^\ast _\alpha}}  
\newcommand{\poRIcover}{\prec_{\mathcal{R}{\mathfrak{S}}^\ast _\alpha}}
 
\newcommand{\poIcover}{\prec_{{\mathfrak{S}}^\ast _\alpha}} 
\newcommand{\pocover}{\prec} 
\newcommand{\po}{\preccurlyeq} 
\newlength{\cellsize}
\newcommand\tableau[1]{
\vcenter{
\let\\=\cr
\baselineskip=-16000pt
\lineskiplimit=16000pt
\lineskip=0pt
\halign{&\tableaucell{##}\cr#1\crcr}}}

\newcommand{\tableaucell}[1]{{%
\def \arg{#1}\def \void{}%
\ifx \void \arg
\vbox to \cellsize{\vfil \hrule width \cellsize height 0pt}%
\else
\unitlength=\cellsize
\begin{picture}(1,1)
\put(0,0){\makebox(1,1)[c]{$#1$}}
\put(0,0){\line(1,0){1}}
\put(0,1){\line(1,0){1}}
\put(0,0){\line(0,1){1}}
\put(1,0){\line(0,1){1}}
\end{picture}%
\fi}}

\DeclareMathOperator{\comp}{comp}

\DeclareMathOperator{\set}{set}
\newcommand{\dI}{\mathfrak{S}^*}
\newcommand{\rdI}{\mathcal{R}\mathfrak{S}^*}

\DeclareMathOperator{\QSym}{QSym}
\DeclareMathOperator{\Des}{Des}
\DeclareMathOperator{\DesI}{Des_{\dI}}

\DeclareMathOperator{\rw}{rw}

\DeclareMathOperator{\inv}{inv}

\DeclareMathOperator{\chr}{\mathrm{ch}}  
\DeclareMathOperator{\bA}{\overline{\mathcal{A}}} 

\newcommand{\Qsym}{\ensuremath{\operatorname{QSym}}}
\newcommand{\SIT}{\ensuremath{\operatorname{SIT}}}
\newcommand{\SET}{\ensuremath{\operatorname{SET}}} 
\newcommand{\NSET}{\ensuremath{\operatorname{NSET}}}
\newcommand{\pia}{\pi _i^a}

\newcommand{\hn}{H_n(0)}

\newcommand{\ninv}{\mathrm{inv}} 

\newcommand{\suchthat}{\;|\;}

\newcommand{\spam}{\operatorname{span}}

\DeclareFontFamily{U}{rcjhbltx}{}  
\DeclareFontShape{U}{rcjhbltx}{m}{n}{<->rcjhbltx}{}
\DeclareSymbolFont{hebrewletters}{U}{rcjhbltx}{m}{n}
\DeclareMathSymbol{\shin}{\mathord}{hebrewletters}{152}

\author[Lafreni\`ere, Orellana,  Pun, Sundaram, van Willigenburg, Whitehead McGinley]{Nadia Lafreni\`ere,  Rosa Orellana, Anna Pun, Sheila Sundaram, Stephanie van Willigenburg and Tamsen Whitehead McGinley}

\address{Nadia Lafreni\`ere: Concordia University, \svw{Montr\'{e}al, QC, Canada}} 
\email{nadia.lafreniere@concordia.ca}
\address{Rosa Orellana: Dartmouth College, Hanover, NH, USA}
\email{rosa.c.orellana@dartmouth.edu}
\address{Anna Pun: Baruch College, New York, NY, USA}
\email{annapunying@gmail.com}
\address{Sheila Sundaram: University of Minnesota, Minneapolis, MN, USA}
\email{shsund@umn.edu}
\address{Stephanie van Willigenburg: University of British Columbia, Vancouver, BC, Canada}
\email{steph@math.ubc.ca}
\address{\svw{Tamsen Whitehead {McGinley:}} Santa Clara University, Santa Clara, CA, USA}
\email{tmcginley@scu.edu}

\title[The skew immaculate Hecke poset]{The  skew immaculate Hecke poset and 0-Hecke modules}
\date{\today}

\begin{document}

\subjclass{05E05, 05E10, 06A07, 16T05, 20C08.}
\keywords{Branching rule,  dual immaculate function, {quasisymmetric function,} row-strict dual immaculate function, skew 0-Hecke  module, immaculate Hecke poset.}

\begin{abstract}{ The immaculate Hecke poset was introduced and investigated by  Niese, Sundaram, van Willigenburg, Vega and Wang, who established the full poset structure, and determined modules for the 0-Hecke algebra action on immaculate and row-strict immaculate tableaux.
 In this paper, we extend their results by introducing the skew immaculate Hecke poset. We investigate the poset structure, and construct  modules for the 0-Hecke algebra action on skew immaculate and skew row-strict immaculate tableaux, thus showing that the skew immaculate Hecke poset  captures representation-theoretic information analogous  to the immaculate Hecke poset.  We also describe branching rules for the resulting skew modules. }
\end{abstract}

\maketitle
\tableofcontents

\section{Introduction}

This work continues the investigations of two papers \cite{NSvWVW2023, NSvWVW2024}  concerned with  the discovery of the row-strict dual immaculate function, and the construction of a \svw{0-Hecke module} whose quasisymmetric characteristic is this function. All of this revolves around  standard immaculate tableaux, which were first defined by Berg, Bergeron, Saliola, Serrano and Zabrocki in \cite{BBSSZ2014}. Niese, Sundaram, van Willigenburg, Vega and Wang showed  in the paper \cite{NSvWVW2024}  that the poset induced by the action of the 0-Hecke algebra on standard immaculate tableaux  contains a wealth of information about various 0-Hecke modules and quotient modules whose quasisymmetric characteristics can be defined purely combinatorially. \svw{These quasisymmetric characteristics are}  the (row-strict) dual immaculate functions and the (row-strict) extended Schur functions.  By establishing that the immaculate Hecke poset is bounded, the authors were able to show in \cite{NSvWVW2024} that many of the associated modules are cyclic and  indecomposable.

The present paper  extends this program  to the 
 \emph{skew immaculate Hecke poset}, 
   the natural generalisation of the immaculate Hecke poset to \emph{skew} standard immaculate tableaux defined in \cite{BBSSZ2014} and \cite{NSvWVW2023}. 
The extended  Schur functions defined  by Assaf and Searles in \cite{AS2019},  and their   row-strict counterparts defined in \cite{NSvWVW2024}, also have skew versions.  We show  that the  skew immaculate Hecke poset is ranked 
and bounded.  
We then examine the module associated  to the skew row-strict dual immaculate function,  and the submodule associated to the skew 
 row-strict extended Schur function; these are  the skew analogues of the corresponding  functions  defined in \cite{NSvWVW2024}.  
  \svw{As in} \cite{NSvWVW2024}, the technical arguments are cast primarily in terms of the row-strict 0-Hecke modules.   By reversing the arrows in the skew immaculate Hecke poset,   we deduce analogous results  for  the modules associated to the skew dual immaculate  functions.   In contrast to the ordinary (non-skew) case in \cite{NSvWVW2024}, it turns out that the subposet of the skew immaculate Hecke poset corresponding to the skew standard extended tableaux  does not always have a unique minimal element for the row-strict dual immaculate action, and hence for this action   the skew extended 0-Hecke module  is in general not cyclic. 

The paper is organised as follows.  \Cref{sec:background} contains a summary of the necessary prerequisites, encapsulated in Table~\ref{table:All4Imm}. \Cref{sec:quasisym-charS} summarises the pertinent combinatorial information from \cite{NSvWVW2023} on skew immaculate tableaux and their quasisymmetric generating functions. Also defined here are the skew analogues of the extended and row-strict extended tableaux and (extended, row-strict extended) Schur \svw{functions.}  \Cref{sec:Skew-poset-moduleS} examines the immaculate and row-strict immaculate 0-Hecke actions on the set of skew standard immaculate tableaux, culminating in the construction of 0-Hecke skew modules (\Cref{{the:4skewmodules}}) whose quasisymmetric characteristics are the 
 generating functions described in \Cref{sec:quasisym-charS}. 
\Cref{sec:Skew-branching} follows Tewari and van Willigenburg \cite{TvW2015} to derive branching rules for the skew modules. Finally  \Cref{sec:Skew-Hecke-poset}  establishes that the skew immaculate Hecke poset is almost always bounded above and below,  thus giving cyclic generators (\Cref{thm-rdI-skew}, \Cref{thm-dI-skew}) for the appropriate 0-Hecke \svw{skew modules.}  The notable and interesting exception is the subposet of skew standard extended tableaux, which does not in general admit a unique bottom element, as can be seen in Figure~\ref{fig:SkewSETPoset}. 

\section{Background}\label{sec:background}

We now recall the relevant algebra and combinatorics that will be useful to us in due course.

\subsection{Quasisymmetric functions}\label{sec:quasisym}
We recall the pertinent definitions below, and refer the reader to \cite{LMvW2013} for further background on quasisymmetric functions.

A {\emph{composition}} of $n$ is a  sequence of positive  integers $\alpha = (\alpha_1, \alpha_2, \ldots,\alpha_k)$ summing to $n$. {For convenience we sometimes denote $k$ by $\ell(\alpha)$, and $n$ by $|\alpha|$.}
Write $\alpha\vDash n$ for a composition $\alpha = (\alpha_1, \alpha_2,\ldots,\alpha_k)$ of $n$, {and denote by $\emptyset$ the empty composition of 0}.

It is well known that compositions of $n$ are in bijection with subsets of $\{1,2,\ldots,n-1\} =[n-1]$. 
For a composition $\alpha\vDash n$, the corresponding set is $\set(\alpha) = \{\alpha_1,\alpha_1+\alpha_2,\ldots,\alpha_1+\cdots+\alpha_{k-1}\}$. 
Given a subset $S=\{s_1,s_2,\ldots,s_j\}$ of $\{1,\ldots,n-1\}$, the corresponding composition of $n$ is $\comp(S)=(s_1,s_2-s_1,\ldots,s_j-s_{j-1},n-s_j)$. 

A function $f\in \bQ[[x_1,x_2,\ldots]]$ is {\em quasisymmetric} if the coefficient of $x_1^{\alpha_1}x_2^{\alpha_2}\cdots x_k^{\alpha_k}$ is the same as the coefficient of $x_{i_1}^{\alpha_1}x_{i_2}^{\alpha_2}\cdots x_{i_k}^{\alpha_k}$ for every $k$-tuple of positive integers  $(\alpha_1, \alpha_2,\ldots,\alpha_k)$ and 
subscripts $i_1<i_2<\cdots <i_k$.  The set of all quasisymmetric functions forms a ring graded by degree, $\QSym = \bigoplus_{{n \geq0}} \QSym_n$, where each $\QSym_n$ is a vector space over the field  $\bQ$ of rationals, with bases indexed by compositions of $n$. 

Given a composition $\alpha=(\alpha_1,\alpha_2,\ldots,\alpha_k)$ of $n$, the {\em fundamental quasisymmetric function} indexed by $\alpha$ is 
\[F_\alpha = \sum_{\substack{i_1\leq i_2\leq \cdots \leq i_n\\i_j<i_{j+1} \text{ if } j\in\set(\alpha)}} x_{i_1}x_{i_2}\cdots x_{i_n}.\]
The set 
$\{F_\alpha:\alpha\vDash n\}$ is a basis for $\QSym_n$, called the \emph{fundamental} basis.

The {\em complement} of a composition $\alpha \vDash n$ \svw{is the composition} $\alpha^c= \comp(\set(\alpha)^c) \vDash n$.
 In $\QSym$ we have the involutive automorphism   $\psi,$ defined on the fundamental basis by
\begin{equation}
\psi(F_\alpha) = F_{\alpha^c} \label{eqn:psiF}.
\end{equation}

The \emph{diagram} of a composition $\alpha$ is a geometric 
depiction of $\alpha$  as left-justified cells, with 
$\alpha_i$ cells in the $i$th row from the bottom. {Given two compositions \svw{$\alpha, \beta$,} we say that $\beta \subseteq \alpha$ if $\beta _j \leq \alpha _j$ for all $1\leq j \leq \ell (\beta) \leq \ell (\alpha)$, and given $\alpha, \beta$ such that $\beta \subseteq \alpha$, the \emph {skew diagram}  $\alpha / \beta$ is the array of \svw{cells} in $\alpha$ but not $\beta$ when $\beta$ is placed in the bottom-left corner of $\alpha$.  } For example, the diagram of $\alpha={(4,2,3)}$  { and $\alpha/\beta$ where $\beta = (2,1)$ are the following, respectively.}
\[\tableau{\phantom{} &\phantom{} &\phantom{} \\\phantom{} &\phantom{}\\ \phantom{} &\textbf{} &\phantom{} & \phantom{}} \qquad {\tableau{\phantom{} &\phantom{} &\phantom{} \\  &\phantom{}\\   &  &\phantom{} & \phantom{}} }\]

\begin{definition}\cite[Definition 2.1]{BBSSZ2015}
Let $\alpha\vDash n$. An \emph{immaculate tableau} of {\emph{shape}} $\alpha$ is a filling $D$ of the cells of the diagram of $\alpha$ with positive integers such that 
\begin{enumerate}[itemsep=1pt]
\item {the} leftmost column entries strictly increase from bottom to top;
\item {the} row entries weakly increase from left to right. 
\end{enumerate}
A \emph{standard}  immaculate tableau  of shape $\alpha\vDash n$ is an immaculate tableau that is filled with the distinct entries $1,2,\ldots,n.$ \svw{A \emph{standard  extended tableau}  of shape $\alpha\vDash n$ is a standard immaculate tableau where the column entries in \emph{every} column strictly increase from bottom to top.}
\end{definition}

\begin{definition} \label{def:dIfunction}
The {\em dual immaculate function} indexed by $\alpha\vDash n$ is $\dI_\alpha = \sum_D x^D$, summed over all immaculate tableaux {$D$} of shape $\alpha$, \svw{with} {$x^D=x_1^{d_1}x_2^{d_2}\cdots $}, 
 \svw{where}  $d_i$ is the number of $i$'s in the tableau $D$.
\end{definition}
\begin{theorem}\label{the:qsbasis}\cite[Definition~2.3, Proposition~3.1]{BBSSZ2015}  
The set $\{\dI_\alpha\}_{\alpha\vDash n}$ is a basis for $\Qsym_n.$
Given a standard immaculate tableau $T$, its $\dI$-{\em descent set}  $\DesI(T)$ of $T$ is 
$$\DesI(T)=\{i: i+1 \text{ appears strictly above }i \text{ in } T \}.$$Then 
$$\dI_\alpha = \sum_{\svw{T}} F_{\comp(\DesI(T))}$$summed  over all standard immaculate tableaux {$T$} of shape $\alpha$.  
\end{theorem}
For example, if $\alpha=(1,2),$ the unique standard immaculate tableau $\tableau{2&3\\1}$ 
 has $\dI$-descent set $\{1\}$; {thus} $\dI_{(1,2)}={F_{\comp( \{1\})}}=F_{(1,2)}.$
 \begin{definition}\cite{NSvWVW2023}\label{def:rdI} Let $\alpha\vDash n$.  A {\emph{row-strict immaculate tableau} of \emph{shape}} $\alpha$ is a filling $U$ of the diagram of $\alpha$ with positive integers such that 
\begin{enumerate}[itemsep=1pt]
\item {the} leftmost column entries weakly increase from bottom to top;
\item {the} row entries strictly increase from left to right.
\end{enumerate}
 {A \emph{standard}  row-strict immaculate tableau  of shape $\alpha\vDash n$ is a row-strict immaculate tableau that is filled with the distinct entries $1,2,\ldots,n.$}
 
The {\em row-strict dual immaculate function} indexed by $\alpha \vDash n$ is $ \rdI_\alpha = \sum_U x^U$   summed  over all row-strict immaculate tableaux $U$ of shape $\alpha$, \svw{with} {$x^U=x_1^{u_1}x_2^{u_2}\cdots $}, 
\svw{where}  {$u_i$} is the number of $i$'s in the tableau $U$. 
\end{definition}
Note that standard row-strict  immaculate tableaux coincide with  standard immaculate tableaux. 
We now establish some notation and definitions that include these tableaux.
\begin{definition}\label{def:SIT-SET-SIT*-NSET} Let $\alpha\vDash n$. 
\begin{enumerate}
    \item A {\emph{standard tableau} of \emph{shape}} $\alpha$ is an arbitrary bijective filling of the diagram of $\alpha$ with the distinct entries $\{1,2,\ldots,n\}$. 
    \item $\SIT(\alpha)$ is the set of standard immaculate tableaux of shape $\alpha$; i.e., the subset of standard tableaux where the {leftmost} column entries increase from bottom to top, and all row entries increase from left to right.
    \item $\SET(\alpha)$ \svw{is the set of all standard extended tableau of shape $\alpha$; i.e., the subset of $\SIT(\alpha)$ where all column entries increase from bottom to top.}
\end{enumerate}
 \end{definition}
  
We also have an analogous theorem to Theorem~\ref{the:qsbasis} for row-strict dual immaculate functions. 

\begin{theorem}\cite{NSvWVW2023} \label{thm:rsfunddecomp}  
The set $\{ \rdI_\alpha\}_{\alpha\vDash n}$ is a basis for $\QSym_n$. Given a standard immaculate tableau $T$, its  $\rdI$-descent set  $\Des_{\rdI}(T)$ of $T$ is $$\Des_{\rdI}(T)=\{i:i+1\text{ is weakly below } i \text{ in } T\}.$$Then $$\rdI_\alpha = \sum_{T} F_{\comp(\Des_{\rdI}(T))}$$ 
summed over all standard  immaculate tableaux {$T$} of shape $\alpha$.
\end{theorem}
For example, if $\alpha=(1,2),$ the unique standard immaculate tableau $\tableau{2&3\\1}$ 
has $\rdI$-descent set $\{2\}$; thus $\rdI_{(1,2)}={F_{\comp( \{2\})}}=F_{(2,1)}.$

Finally  observe that since the descent sets $\Des_{\dI}(T)$ and $\Des_{\rdI}(T)$ are complements of each other, \svw{by \eqref{eqn:psiF}} we have $$\rdI_\alpha=\psi(\dI_\alpha).$$

\subsection{The 0-Hecke algebra}\label{sec:0-Hecke-algebra}
Recall that the symmetric group $S_n$ can be defined via generators $s_i, 1\le i\le n-1,$ that satisfy\begin{center}${s_i}^2 =1;  \quad  s_i s_{i+1}s_i =s_{i+1}s_is_{i+1};\quad    s_is_j =s_js_i, \ |i-j|\ge 2.$\end{center}
\begin{definition}\cite{MathasHecke1999} Let $\bK$ be any field.  The {\emph{0-Hecke algebra}} $H_n(0)$ is the $\bK$-algebra with 
generators $\pi_i, 1\le i\le n-1$ and relations 
\begin{center}$ {\pi_i}^2 =\pi_i; \quad 
          \pi_i\pi_{i+1}\pi_i =\pi_{i+1}\pi_i\pi_{i+1};\quad
              \pi_i\pi_j =\pi_j\pi_i, \ |i-j|\ge 2.$\end{center}\end{definition}
              
The algebra $H_n(0)$ has dimension $n!$ over  $\bK$, with basis elements $\{\pi_\sigma: \sigma\in S_n\},$ where $\pi_\sigma=\pi_{i_1}\cdots \pi_{i_m}$ if $\sigma=s_{i_1}\cdots s_{i_m}$  is a reduced word.  This is well-defined by standard Coxeter group theory, see \cite{MathasHecke1999}.

It is known   \cite{PamelaBromwichNorton1979} that the 0-Hecke algebra {$H_n(0)$} admits  precisely $2^{n-1}$  simple modules $L_\alpha$, 
one for each composition $\alpha\vDash n$, and all {are} one-dimensional.  The module $L_\alpha$ is defined as $L_\alpha = \spam\{v_\alpha\}$ where for each $i=1, \ldots, n$, we have 
$\pi_i(v_\alpha)=\begin{cases} 0 , &  i\in \set(\alpha),\\
                                                v_\alpha, &\text{otherwise}.\end{cases}$

For the remainder of this paper we take the ground field to be the field  $\mathbb{C}$ of complex numbers.   

The well-known  Frobenius characteristic defined on the Grothendieck ring of the symmetric groups has the following analogue for finite-dimensional $\hn$-modules.  See  \cite{DKLT1996} for further details.  
\begin{definition} \cite[Section 5.4]{KrobThibon1997} Let $M$ be a finite-dimensional $\hn$-module;  let $M=M_1\supset M_2\supset \cdots \supset M_k\supset M_{k+1}=\{0\}$ be a composition series of submodules for $M$, where  each successive quotient $M_i/M_{i+1}$ is simple, and thus equal to $L_{\alpha^i}$ for some composition $\alpha^i\vDash n$. The \emph{quasisymmetric characteristic} of the module $M$, $\chr(M)$, is then defined to be the quasisymmetric function  equal to the sum of fundamental quasisymmetric functions $\sum_{i=1}^k F_{\alpha^i}$.
In particular $\chr(L_\alpha)=F_\alpha$ for each $\alpha\vDash n$.
\end{definition}

\subsection{0-Hecke modules  on   $\SIT(\alpha)$}\label{sec:0-Hecke-moduleS-straight}
We begin with a general proposition underlying the arguments in much of the literature on the \svw{subject.} 

\begin{proposition}\label{prop:poset-to-composition-serieS}  
Let $V$ be a 0-Hecke module with finite $\mathbb{C}$-basis $\mathcal{T}$, 
and suppose  $\hn$ acts on the basis $\mathcal{T}$ so that
 for each generator $\pi_i$ of $\hn$, either 
 $\pi_i(T)\in \mathcal{T}$ or $\pi_i(T)=0$.
 Define a  relation  on $\mathcal{T}$ by setting, for $S,T\in \mathcal{T}$,   \[S\po T\] if 
there is a sequence of $\hn$-generators $\pi_{s_1},\ldots, \pi_{s_r}$ such that \[T=\pi_{s_1}\cdots \pi_{s_r}(S).\] 
\svw{Suppose further that the relation $\po$ makes 
$\mathcal{T}$ into a partially ordered set}. 

\begin{enumerate}
    \item 
Let $T_1 \pocover^t\cdots \pocover^t T_m $ be an arbitrary  total order on $\mathcal{T}$ that extends the partial order defined by $\po$. 
Let $V_{T_{m+1}}=\{0\}$, and let $V_{T_i}$ be the $\mathbb{C}$-linear span 
\[V_{T_i}=\mathrm{span} \{T_j: T_i{\po^t }T_j\}, \quad 1\le i\le {m.}\]
Then each $V_{T_i}$ is an $\hn$-module, and the filtration 
\[\{0\}=V_{T_{m+1}}\subset V_{T_m}\subset\cdots \subset V_{T_1}=V\]
has one-dimensional, hence simple, quotient modules, i.e., it is a composition series for $V_{T_1}=V$. 
The quasisymmetric characteristic of $V$ is then 
\[\chr(V)=\sum_{i=2}^{m+1} \chr(V_{T_{i-1}}/V_{T_i}).\]
In particular, for each $i$,  $\chr(V_{T_{i-1}}/V_{T_i})$ is a fundamental quasisymmetric function $F_{\alpha^i}$ for some composition $\alpha^i\vDash n$.  
\item
 If the partial order $\po$ admits a unique least element $T_{bot}$, then the module $V=\mathrm{span} \{T_j\}_{j=1}^m$ is cyclic and generated by  $T_{bot}$.
 \item If the Hasse diagram of the partial order $\po$ is disconnected, then the module $V=\mathrm{span} \{T_j\}_{j=1}^m$ decomposes into a sum of at least two proper \svw{submodules.}
\end{enumerate}
\end{proposition}
\begin{proof} For Part (1), see, e.g., \cite[Section 5]{TvW2015} or \cite[Theorem~5.5]{NSvWVW2024}.  
Parts (2) and (3) \svw{follow} from the definition of the partial order $\po$.
\end{proof}

The following four  descent sets  were considered in \cite{NSvWVW2024} {for $T\in \SIT(\alpha)$}.  We have already encountered the first two;  the other two complete all variants.
\begin{enumerate}
    \item {\cite[Definition 3.20]{BBSSZ2014}} The {\emph{dual immaculate descent set} } 
    $$\Des_{\dI}(T)=\{i: i+1 \text{ is in a row strictly above $i$ in $T$}\};$$    
    \item {\cite[Defintion 3.4]{NSvWVW2023}} The {\emph{row-strict dual immaculate descent set }}
    $$\Des_{\rdI}(T)=\{i: i+1 \text{ is  in a row weakly below $i$ in $T$}\};$$
    \item {\cite[Definition 9.1]{NSvWVW2024}} The {\emph{$\mathcal{A}^*$-descent set}} $$\Des_{\mathcal{A}^*}(T)=\{i: i+1 \text{ is  in a row strictly below $i$ in $T$}\};$$
    \item {\cite[Definition 9.2]{NSvWVW2024}} The {\emph{$\bA^*$-descent set}} $$\Des_{\bA^*}(T)=\{i: i+1 \text{ is  in a row weakly above $i$ in $T$}\}.$$
\end{enumerate}

Each of these descent sets defines a 0-Hecke action of $\SIT(\alpha)$. 
The following theorems summarise the main results of \cite{NSvWVW2024}.  See {\cite{BBSSZ2015} for the impact of the $\dI$-action.}
 The remaining actions appear in  \cite{NSvWVW2024}, where a unified approach is given.

\begin{theorem}\label{thm:0-Hecke-action} {\cite{BBSSZ2015, NSvWVW2024}} Let $\alpha \vDash n$. Each of the four descent sets $\Des_{\dI}$, $\Des_{\rdI}$, $\Des_{\mathcal{A}^*}$, $\Des_{\bA^*}$  defines a cyclic $\hn$-module on $\SIT(\alpha)$.  
 Let $\pi^a$ denote any one of these four actions, and 
 denote the descent set of {$T\in \SIT(\alpha)$} in each case simply by $\Des_a(T)$. 
  Also let $s_i$ be the operator \svw{that} switches the entries $i$, $i+1$ in a tableau $T$.

In each case the action of a generator $\pi^a_i$ is defined as
\begin{equation}\label{eqn:defn-pi(T)}
\pi_i^a(T)=\begin{cases} T, &  i\notin \Des_a(T),\\
                        s_i(T), &  i\in \Des_a(T) \text{ and } s_i(T)\in \SIT(\alpha),\\
                        0, & otherwise.
 \end{cases}
 \end{equation} 
\end{theorem}

When the need arises to be more specific, we  will refer to the actions defined by $\pi_i^a$ as $a$-actions, where $a$ is one of $\dI$, $\rdI$, $\mathcal{A}^*$ or $\bA^*$.  

Next we define  \svw{special} tableaux so that we can concretely describe the poset in the theorem below and all  entries in Table~\ref{table:All4Imm} thereafter.

\begin{definition}\label{def:maxminels} For $\alpha \vDash n$ we define the following special standard tableaux in $\SIT(\alpha)$. 
\begin{itemize}
\item   \cite[Definition 6.3]{NSvWVW2024} $S^0_{\alpha}$ is the  tableau in which the cells of the leftmost column of $\alpha$ are filled  with 
$1,\ldots, \ell(\alpha)$, increasing bottom to top;  then the remaining cells \svw{are filled} by rows, \emph{top to bottom}, left to right with consecutive integers starting at $\ell(\alpha) +1$ and ending at $n.$  

\item  \cite[Definition 6.9]{NSvWVW2024} $S^{row}_{\alpha}$ is the \emph{row superstandard}  tableau  whose rows are filled left to right, \emph{bottom to top}, beginning with the bottom row and moving up, using the numbers $1,2, \ldots, n$ taken in consecutive order.  

\item \cite[Definition 7.7]{NSvWVW2024} $S^{col}_{\alpha}$ is the \emph{column superstandard}  tableau  whose columns are  filled   bottom to top,  left to right, beginning with the leftmost column and moving   right, using the numbers $1,2,\ldots, n$  taken in consecutive order.  
\end{itemize}
\end{definition}

\begin{example}\label{ex:posetminmax}
$$S^0_{(3,1,2)} = \tableau{3&4&5\\2\\1&6}  \qquad S^{row}_{(3,1,2)} = \tableau{4&5&6\\3\\1&2} \qquad S^{col}_{(3,1,2)} = \tableau{3&5&6\\2\\1&4}$$
\end{example}

\svw{The next theorem highlights the importance of these tableaux.}

\begin{theorem}\cite{NSvWVW2024}\label{thm:all4actions}
Each of the  four actions  in Theorem~\ref{thm:0-Hecke-action} defines a partial order on the set $\SIT(\alpha)$, and all four actions then define the same poset $P\rdI_\alpha$, {called} the \emph{immaculate Hecke poset}.

The poset $P\rdI_\alpha$ is graded and bounded with a unique minimal element $S^0_\alpha$ and a unique maximal element $S^{row}_\alpha$. The vector space spanned by the set $\SIT(\alpha)$ then becomes 
\begin{itemize}
    \item  \cite[Theorems 6.8 and 6.15]{NSvWVW2024} a cyclic indecomposable 0-Hecke module $\mathcal{V}_\alpha$ for the $\rdI$-action, with generator $S^0_\alpha$.  Its quasisymmetric characteristic is the row-strict dual immaculate function $\rdI_\alpha$;
    \item {\cite[Theorem 3.5 and Lemma 3.10]{BBSSZ2015}} a cyclic indecomposable 0-Hecke module $\mathcal{W}_\alpha$ for the $\dI$-action, with generator $S^{row}_\alpha$.  Its quasisymmetric characteristic is the dual immaculate function $\dI_\alpha$;
    \item \cite[Theorem 9.3]{NSvWVW2024} a \svw{cyclic} 0-Hecke module $\mathcal{A}_\alpha$ for the $\mathcal{A}^*$-action, with generator $S^0_\alpha$.  Its quasisymmetric characteristic is $\mathcal{A}^*_\alpha=\sum_{T\in\SIT(\alpha)} {F_{\comp(\Des_{\mathcal{A}^*}(T))}}$;
    \item \cite[Theorem 9.4]{NSvWVW2024} a \svw{cyclic}  0-Hecke module $\mathcal{\bA}_\alpha$ for the $\bA^*$-action, with generator $S^{row}_\alpha$.   Its quasisymmetric characteristic is $\mathcal{\bA}^*_\alpha=\sum_{T\in\SIT(\alpha)} {F_{\comp(\Des_{\mathcal{\bA}^*}(T))}}$.
\end{itemize}
\end{theorem}

 \svw{\Cref{table:All4Imm} gives a  summary of other useful  results in \cite{BBSSZ2015, NSvWVW2024} for quick reference later. Additionally some useful results in \cite{AS2019, CFLSX2014, NSvWVW2024, S2020} are given for the extended and row-strict extended Schur functions, which we discuss in the next section.}
\begin{table}[h!]
\caption{The four descent sets and resulting tableaux}
\begin{center}
\scalebox{0.8}{
\begin{tabular}{|c|c|c|c|c|l|}
\hline
{Tableaux $\rightarrow$}  &  $\dI_\alpha\ $  &  {$\rdI_\alpha$}& ${\mathcal{A}_\alpha^*}$ &${\bA_\alpha^*}$ \T \\
 &  Dual immaculate &  { Row-strict dual imm.} & $\chr(\mathcal{A}_\alpha)$ & $\chr(\bA_\alpha)$  \\
[2pt]\hline\hline
 \svw{{Leftmost} Col}  &strict $\nearrow$ &{weak $\nearrow$ } &{weak $\nearrow$ }&strict $\nearrow$\\
\svw{bottom to top} & & & &\\[2pt]\hline
 \svw{Rows} 
& weak $\nearrow$&{strict  $\nearrow$}&{weak $\nearrow$ } & strict $\nearrow$\\
 \svw{left to right} & & & & \\[2pt]\hline
Descents & $i$ such that & $i$ such that & $i$ such that & $i$ such that \\
  &$ i+1 \text{ strictly above } i$ & $i+1 \text{ weakly below } i$ 
&$  i+1 \text{ strictly below } i$ & $  i+1 \text{ weakly above } i$\\[2pt]\hline\hline
 $\pi_i(T)=T$ & $ i+1 \text{ weakly below } i$ 
& {$ i+1 \text{ strictly above } i$} &{$ i+1 \text{ weakly above } i$}  
                                                    &$ i+1 \text{ strictly below } i$
 \\[2pt]\hline
$\pi_i(T)=0$ & $i, i+1$ in  & {$i, i+1$ in same row}& NEVER & $i, i+1$ in {leftmost} \\
& {leftmost column} & & & column or  in same row \\[2pt]\hline
 $T, s_i(T)$ standard, &$ i+1 $\text{ strictly above }$ i$,& {$ i+1 \text{ strictly below } i$}  &  {same as $\rdI$}&  same as $\dI$\\[2pt]
and $\pi_i(T)=s_i(T)$ & $i,i+1$ NOT both  & \phantom{{ $i, i+1$ NOT in same row}} &&\\
&in {leftmost column} & & &\\[2pt]\hline\hline
 {Partial order} &Poset $\svw{P\dI_{\alpha}}$&{ Poset $\svw{P\rdI_{\alpha}}$} &{ Poset $\svw{P\rdI_{\alpha}}$} &Poset $\svw{P\dI_{\alpha}}$\\
 {on $\SIT(\alpha)$} &{$\simeq \svw{P\rdI_{\alpha}}$} & $=[S^{0}_\alpha, S^{row}_\alpha]$ & $=[S^{0}_\alpha, S^{row}_\alpha]$&$\simeq \svw{P\rdI_{\alpha}}$\\[2pt]\hline\hline
Cover relation & $S\!\poIcover\!  T\!$ & {$S\!\poRIcover\!  T$}  &&\\
& $\iff\!  S=\pi_i^{\dI}(T)$ &{$\!\iff\!  T=\pi_i^{\rdI} (S)$} &  {same as $\rdI$}&  same as $\dI$ 
\\[2pt]\hline
Module& top element:   & {bottom  element:} &{bottom  element:} &top element:\\
generated by & $\mathcal{W}_\alpha=\langle S^{row}_\alpha\rangle$ & {$\mathcal{V}_\alpha=\langle S^0_\alpha\rangle$} &$\mathcal{A}_\alpha=\langle S^0_\alpha\rangle$ 
&$ \bA_\alpha=\langle S^{row}_\alpha\rangle$ \\[2pt]\hline
 Indecomposable? & Yes & {Yes}&\svw{Unknown}&\svw{Unknown}\\[2pt]\hline\hline
Extended & $\mathcal{E}_\alpha$& {$\psi(\mathcal{E}_\alpha)=\mathcal{R}\mathcal{E}_\alpha$}& ${\mathcal{A}_{\SET({\alpha})}^*}$& 
 ${\bA_{\SET({\alpha})}^*}$  \T \\
Schur function &$\mathcal{E}_\lambda=s_\lambda$ &{$\mathcal{R}\mathcal{E}_\lambda=s_{\lambda^t}$} & &  \\[2pt]\hline
Module & cyclic $\langle S^{row}_\alpha\rangle$ & {{cyclic} $\langle S^{col}_\alpha\rangle$} &{\svw{cyclic} $\langle S^{col}_\alpha\rangle$} &\svw{cyclic} $\langle S^{row}_\alpha\rangle$  \T \\
 &indecomp. & {indecomp.} &   &   \\[2pt]
Basis {$\SET(\alpha)$} & quotient of $\mathcal{W}_\alpha$ & {submodule of $\mathcal{V}_\alpha$} & {submodule of $\mathcal{A}_\alpha$} & quotient of $\bA_\alpha$ \\[2pt]\hline
\end{tabular}
}
\end{center}
\label{table:All4Imm}
\end{table}

\section{Skew immaculate  tableaux and quasisymmetric functions}\label{sec:quasisym-charS}

{We now broaden our horizons, moving from diagrams to skew diagrams.}
\begin{definition}\label{def:skew-SIT} 
{Let $\alpha,\beta$ be compositions with $\beta \subseteq \alpha$. A \emph{skew standard immaculate tableau} of {\emph{shape}} $\alpha/\beta$ is a filling $T$ of the cells of the skew diagram $\alpha/\beta $ with distinct entries $1, 2, \ldots , |\alpha/\beta|=|\alpha|-|\beta|$ such that 
\begin{enumerate}[itemsep=1pt]
\item the leftmost column entries that belong to $\alpha$ but not $\beta$ increase from bottom to top;
\item the row entries  increase from left to right. 
\end{enumerate}
}
Let $\SIT(\alpha/\beta)$ denote the set {of}  skew standard  immaculate tableaux of shape $\alpha/\beta$.  
\end{definition}

{Example~\ref{ex:skewSIT-to-straightSIT} gives an example of a skew standard immaculate tableau.}

\begin{definition}\label{def:skew-SET}  {Let $\alpha,\beta$ be compositions with $\beta \subseteq \alpha$. Then} define $\SET(\alpha/\beta)$ to be the subset of $\SIT({\alpha/\beta})$ consisting of all tableaux $T$ where {all columns entries increase bottom to top,} \svw{that is, the set of all \emph{skew standard extended tableaux} of \emph{shape} ${\alpha/\beta}$.}\end{definition}

{Note that  \svw{when} $\beta=\emptyset$ we have that $\SIT({\alpha/\beta})$ and $\SET(\alpha/\beta)$ coincide with $\SIT({\alpha})$ and $\SET(\alpha)$.}

\begin{definition} Let $\alpha,\beta$ be compositions with {$\beta \subseteq \alpha$.} Then we define eight sets of {tableaux of shape $\alpha/\beta$} as follows.
 
\svw{In the definitions that follow, the phrase ``{entries} in the {leftmost} column of $\alpha$"  refers to those {entries in the} cells {of} the {leftmost} column of the {skew diagram} $\alpha/\beta$ {that belong to $\alpha$ but not $\beta$.}}
\begin{enumerate}
\item   $\mathcal{T}_{\alpha/\beta}({\text{1st col}<, \text{rows}\le})$ {is  the set of tableaux  with entries in the leftmost column of $\alpha$ strictly increasing  bottom to top, and row entries weakly increasing  left to right.} 
\item  $\mathcal{T}_{\alpha/\beta}({\text{1st col}\le, \text{rows}<})$ {is  the set of tableaux  with entries in the leftmost column of $\alpha$ weakly increasing  bottom to top, and row entries strictly increasing  left to right.}
\item
 $\mathcal{T}_{\alpha/\beta}({\text{cols}<, \text{rows}\le})$ {is  the set of tableaux  with column entries strictly increasing  bottom to top, and row entries weakly increasing  left to right.}
\item
 $\mathcal{T}_{\alpha/\beta}({\text{cols}\le, \text{rows}<})$ {is  the set of tableaux  with column entries weakly increasing  bottom to top, and row entries strictly increasing  left to right.}

\item   $\mathcal{T}_{\alpha/\beta}({\text{1st col}\le, \text{rows}\le})$ {is  the set of tableaux  with entries in the leftmost column of $\alpha$ weakly increasing  bottom to top, and row entries weakly increasing  left to right.}

\item  $\mathcal{T}_{\alpha/\beta}({\text{1st col}<, \text{rows}<})$ {is  the set of tableaux  with entries in the leftmost column of $\alpha$ strictly increasing  bottom to top, and row entries strictly increasing  left to right.}

\item   $\mathcal{T}_{\alpha/\beta}({\text{cols}\le, \text{rows}\le})$ {is  the set of tableaux  with column entries weakly increasing  bottom to top, and row entries weakly increasing  left to right.}

\item  $\mathcal{T}_{\alpha/\beta}({\text{cols}<, \text{rows}<})$ {is  the set of tableaux  with column entries strictly increasing  bottom to top, and row entries strictly increasing  left to right.}
\end{enumerate}
\end{definition}
Skew analogues of the dual immaculate function and the row-strict dual immaculate function were first introduced in \cite{BBSSZ2014} and \cite{NSvWVW2023} respectively and we define them with alternative characterizations and variants.
The first two items {in the theorem} below were established in \cite{NSvWVW2023}.  \svw{The remaining items are proved similarly to the first two items.  The cases $\beta=\emptyset$ were established in \cite[Proposition~7.24, Theorem~9.16]{NSvWVW2024}, and the third and fourth cases are called the \emph{extended Schur function} \cite{AS2019} and \emph{row-strict extended Schur function} \cite{NSvWVW2024}, respectively.} 

{Before we state our theorem, we need the following \svw{concept.} Given a filling $T$ of a (skew) diagram  with entries $\{1,2,\ldots\}$, let 
$x^T$ 
denote the monomial $x_1^{d_1} x_2^{d_2}\cdots$, 
where $d_i$ is the number of entries equal to \svw{$i$ in $T$. }}

\begin{theorem}\label{thm:8flavours-skew-tableaux}  {Let $\alpha, \beta$ be compositions with $\beta \subseteq \alpha$.}
\begin{enumerate}
\item Define the \emph{skew dual immaculate function} as $\dI_{\alpha/\beta}{=}\sum_{T\in\SIT(\alpha/\beta)}  F_{\comp({\Des_{\dI}}(T))}$.  Then 
\[
\dI_{\alpha/\beta}
=\sum_{T\in  \mathcal{T}_{\alpha/\beta}({\text{1st col}<, \text{rows}\le})}  x^T.\]
\item Define the \emph{row-strict skew dual immaculate function} as $\rdI_{\alpha/\beta}{=}\sum_{T\in\SIT(\alpha/\beta)}  F_{\comp({\Des_{\rdI}}(T))}$.  Then 
\[
\rdI_{\alpha/\beta}
=\sum_{T\in  \mathcal{T}_{\alpha/\beta}({\text{1st col}\le, \text{rows}<})}  x^T.\]
\item Define the \emph{skew extended Schur function} as $\mathcal{E}_{\alpha/\beta}{=}\sum_{T\in\SET(\alpha/\beta)}  F_{\comp({\Des_{\dI}}(T))}.$  Then 
\[\mathcal{E}_{\alpha/\beta}
=\sum_{T\in  \mathcal{T}_{\alpha/\beta}({\text{cols}<, \text{rows}\le})} x^T.\]
\item Define the \emph{row-strict skew extended Schur function} as $\mathcal{R}\mathcal{E}_{\alpha/\beta}{=}\sum_{T\in\SET(\alpha/\beta)}  F_{\comp({\Des_{\rdI}}(T))}$.  Then 
\[\mathcal{R}\mathcal{E}_{\alpha/\beta}
=\sum_{T\in  \mathcal{T}_{\alpha/\beta}({\text{cols}\le, \text{rows}<})}  x^T.\]

\item Define $\mathcal{A}^*_{\alpha/\beta}{=}
\sum_{T\in \SIT({\alpha/\beta})} F_{\comp(\Des_{\mathcal{A}^*}(T))}$.  Then 
\begin{equation*}\label{eqn:gf-A-skew}
\mathcal{A}^*_{\alpha/\beta}
=
\sum_{T\in  \mathcal{T}_{\alpha/\beta}({\text{1st col}\le, \text{rows}\le})}  x^T.
\end{equation*}
\item Define $\bA^*_{\alpha/\beta}{=}
\sum_{T\in \SIT({\alpha/\beta})} F_{\comp(\Des_{\bA^*}(T))}$.  Then 
\begin{equation*}\label{eqn:gf-barA-skew}
\bA^*_{\alpha/\beta}
=\sum_{T\in  \mathcal{T}_{\alpha/\beta}({\text{1st col}<, \text{rows}<})}  x^T.
\end{equation*}
\item Define ${\mathcal{A}_{\SET({\alpha/\beta})}^*}{=}\sum_{T\in \SET({\alpha/\beta})} F_{\comp(\Des_{\mathcal{A}^*}(T))}$.  Then 

\begin{equation*}\label{eqn:gf-A-SET-skew}
{\mathcal{A}_{\SET({\alpha/\beta})}^*}
=\sum_{T\in  \mathcal{T}_{\alpha/\beta}(\text{cols}\le, \text{rows}\le)}  x^T.
\end{equation*}
\item   Define ${\bA_{\SET({\alpha/\beta})}^*}{=}
\sum_{T\in \SET({\alpha/\beta})} F_{\comp(\Des_{\bA^*}(T))}$.  Then 
\begin{equation*}\label{eqn:gf-barA-SET-skew}
{\bA_{\SET({\alpha/\beta})}^*}
=\sum_{T\in  \mathcal{T}_{\alpha/\beta}({\text{cols}<, \text{rows}<}) } x^T.
\end{equation*}
\end{enumerate}
\end{theorem} 

In the case where $\alpha$ is a \emph{partition} $\lambda$, namely the parts satisfy $\lambda_1\geq\lambda _2 \geq \cdots >0$ and $\beta = \emptyset$, \svw{\cite{AS2019}} showed that $\mathcal{E}_{\lambda}$ coincides with the Schur function $s_\lambda$, and \cite{NSvWVW2024} showed that $\mathcal{R}\mathcal{E}_{\lambda}$ coincides with the Schur functions $s_{\lambda ^t}$ where $\lambda ^t$ is the transpose of $\lambda$. Similarly we conclude the following, and refer the reader to \cite{LMvW2013} for necessary symmetric function definitions.  

\begin{corollary}\label{cor:skewextendeds} Let $\lambda, \mu$ be partitions with $\mu \subseteq \lambda$. Then 
\begin{enumerate}
\item $\mathcal{E}_{\lambda/\mu}= s_{\lambda / \mu}$,
\item $\mathcal{R}\mathcal{E}_{\lambda/\mu}= s_{(\lambda/\mu)^t}$, 
\end{enumerate}
where $s_{\lambda / \mu}$ is the skew Schur function indexed by the skew diagram $\lambda/\mu$ and $(\lambda/\mu)^t$ is the transpose of $\lambda/\mu$.
\end{corollary}

When $\ell(\alpha)=\ell(\beta)$, we have  simpler expressions for some of the quasisymmetric functions defined above. Here $h_n$ and $e_n$ are respectively the homogeneous and elementary symmetric functions.
\begin{corollary} Let $\alpha, \beta$ be compositions with $\beta \subseteq \alpha$ and $\ell(\alpha)=\ell(\beta)$. Then 
\begin{enumerate}
    \item $
    \dI_{\alpha/\beta}=\prod_i h_{\alpha_i-\beta_i}
    =\svw{{\mathcal{A}^*_{\alpha/\beta},}}
    $
    
    \item $
    \rdI_{\alpha/\beta}=\prod_i e_{\alpha_i-\beta_i} 
    =\svw{\bA^*_{\alpha/\beta}}
    .$
\end{enumerate}    
\end{corollary}

\svw{\begin{proof}Observe that because we have $\ell(\alpha)= \ell(\beta)$, the only condition we need to consider is that on the rows. Therefore,
$$\dI_{\alpha/\beta}=\mathcal{A}^*_{\alpha/\beta}=\prod _i \sum _{j_1\leq \cdots \leq j_{\alpha_i-\beta_i}} x_{j_1}\cdots x_{j_{\alpha_i-\beta_i}}= \prod_i h_{\alpha_i-\beta_i}.$$
Similarly,
$$ \rdI_{\alpha/\beta}={\bA^*_{\alpha/\beta}}=\prod _i \sum _{j_1< \cdots < j_{\alpha_i-\beta_i}} x_{j_1}\cdots x_{j_{\alpha_i-\beta_i}}=\prod_i e_{\alpha_i-\beta_i}.\qedhere$$
\end{proof}}

In \svw{Sections~\ref{sec:Skew-poset-moduleS} and~\ref{sec:Skew-Hecke-poset}} we will use the respective descent sets to construct modules  whose quasisymmetric characteristics are  each of the  functions listed in \Cref{thm:8flavours-skew-tableaux}.

The following result was proved in \cite{NSvWVW2023}, and will be useful later.
\begin{theorem}\label{thm:two-varS}\cite[Theorem~4.11]{NSvWVW2023}
Suppose we have {two} sets of variables, $X$ and $Y$, ordered so that the $X$ alphabet precedes the $Y$ alphabet.  Then 
\begin{equation}\label{eqn:sheila-dI-two-vars}\dI_\alpha(X,Y)= \sum_{\beta\subset \alpha}\dI_\beta (X)  \dI_{\alpha/\beta} (Y),\end{equation}
\begin{equation}\label{eqn:sheila-rdI-two-vars}\rdI_\alpha(X,Y)= \sum_{\beta\subset \alpha}\rdI_\beta (X)  \rdI_{\alpha/\beta} (Y).\end{equation}
\end{theorem}

\section{0-Hecke modules on  $\SIT(\alpha/\beta)$}\label{sec:Skew-poset-moduleS}
As described in Section~\ref{sec:background}, 
 for a composition $\alpha$, 
various 0-Hecke actions on the set of standard immaculate tableaux $\SIT(\alpha)$ of composition shape were considered in \cite{NSvWVW2024}. The goal of this section is to show that each of these actions  can be transferred in a canonical way to an action on  the set of standard immaculate tableaux $\SIT(\alpha/\beta)$  of shape $\alpha/\beta$, where $\alpha/\beta$ is a \emph{skew} diagram.

Let $\alpha\vDash n$, $\beta\vDash m$ with $\beta\subseteq \alpha$.
We describe a canonical identification of  the set $\SIT(\alpha/\beta)$ with a  subset of $\SIT(\alpha)$. We then use this map to transfer any action of $\hn$ on $\SIT(\alpha)$ to an  action of $H_{n-m}(0)$ on $\SIT(\alpha/\beta)$.

\begin{definition}\label{def:map-skew-to-straight}
For $\alpha \vDash n, \beta \vDash m$ and $\beta \subseteq \alpha$, define a map $\phi: \SIT(\alpha/\beta)\rightarrow\SIT(\alpha)$ as follows. If $\svw{T}\in \SIT(\alpha/\beta)$, then $\phi(T)$ is obtained from $T$ by
\begin{itemize}
\item  filling the  cells of $\alpha$ that are not in $\beta$ by adding $m$ to each entry of $T$, and then 
\item  filling the cells of $\alpha$ that are in $\beta$ with the integers $1,\ldots, m$ 
in row superstandard order, that is, with $1, 2, \ldots, \beta_1$ in the first (bottom) row of (the diagram of) $\beta$, $\beta_1+1, \ldots, \beta_1+\beta_2$ in the second row of $\beta$, and so on. 
\end{itemize}
\end{definition} 

\begin{example}\label{ex:skewSIT-to-straightSIT}  Let $\alpha=(2,2,3,2,4)\vDash 13$ and $\beta=(2,1,2)\vDash 5$.
\[\text{Let } T=\tableau{3& \svw{6}&7 &8\\ 2 & \svw{5}\\ \bullet &\bullet &1\\ \bullet  &4\\\bullet &\bullet}\in \SIT(\alpha/\beta).\quad \text{ Then } \phi(T)=\tableau{8& \svw{11} & 12 & 13\\7 & \svw{10}\\ {\bm 4} &{\bm 5} &6\\{\bm 3} & 9\\{\bm 1}&{\bm 2}} \in \SIT(\alpha).\]

\end{example}

Recall from \Cref{thm:all4actions} the  four actions defined on the bounded and ranked poset $P\rdI_\alpha$, defined  respectively by the four descent sets $\Des_{\dI}$, $\Des_{\rdI}$, $\Des_{\mathcal{A}^*}$, $\Des_{\bA^*}$.  
The descent sets for standard immaculate  tableaux of shape  $\alpha/\beta$ are defined in exactly the same manner as for the case $\beta=\emptyset$. In particular, we see that for each type of descent set, and any $T\in \SIT(\alpha/\beta)$, we have, thanks to the map $\phi$, that 
\[i\in \Des(T) \iff i+m\in \Des(\phi(T)).\]
Here we have used $\Des(T)$ to represent any of the four types of descent sets.

\begin{proposition}\label{prop:skew-action} Let $\alpha \vDash n, \beta \vDash m$ with $\beta \subseteq \alpha$. Each of the four $\hn$-actions $\pi^{\dI}$, $\pi^{\rdI}$, $\pi^{\mathcal{A}^*}$, $\pi^{\bA^*}$ on $\SIT(\alpha)$ gives an associated $H_{n-m}(0)$-action, defined by the respective descent set $\Des_{\dI}$, $\Des_{\rdI}$, $\Des_{\mathcal{A}^*}$, $\Des_{\bA^*}$, on the set of  tableaux  $\SIT(\alpha/\beta)$.   Let $\pi^a$ denote any one of these four actions. Denote the descent set of $T\in \SIT(\alpha/\beta)$ in each case by $\Des_a(T)$. Also let $s_i$ be the operator that switches the entries $i$, $i+1$ in a tableau $T$.

In each case the action of a generator $\pi_i^a$ is defined as 
\begin{equation}\label{eqn:defn-skew-pi(T)}
\pi_i^a(T)=\begin{cases} T, &  i\notin \Des_a(T)\\
                        s_i(T), &  i\in \Des_a(T) \text{ and } s_i(T)\in \SIT(\alpha/\beta),\\
                        0, & otherwise.
 \end{cases}
 \end{equation}
\end{proposition}
\begin{proof} This proposition will follow by examining the map $\phi$ of Definition~\ref{def:map-skew-to-straight}. First observe that $\phi$ is injective.  Moreover, given compositions $\beta\subseteq \alpha$,  we also have that any action of $\hn$ on $\SIT(\alpha)$   induces an associated action of  $H_{n-m}(0)$ on $\SIT(\alpha/\beta)$, defined for $1\le i\le n-m-1$ by 
\[\pi^a_i(T)= \pi^a_{i+m}( \phi(T)) \text{ for each } T\in \SIT(\alpha/\beta). \]
The generator $\pi^a_{i+m}$ affects only the entries $i+m, i+m+1$ in the tableau $\phi(T)$, leaving the entries in the cells of $\beta$ fixed. Also we know from Theorem~\ref{thm:0-Hecke-action} that the operators $\{\pi^a_{i+m}: 1\le i\le n-m-1\}$ on the set $\SIT(\alpha)$  satisfy the 0-Hecke relations. It follows that one does indeed have an action of $H_{n-m}(0)$ on $\SIT(\alpha/\beta)$.  
\end{proof}

Now consider either of the two actions $\pi^{\rdI}$ or $\pi^{\mathcal{A}^*}$. Recall from \cite{NSvWVW2024} that the partial order defined on $\SIT(\alpha)$  by the cover relation \[S\poRIcover T \text{ if } T=\pi _i^{\rdI}(S) \quad \textrm{for some $i$}\]
makes $\SIT(\alpha)$ into a ranked and bounded poset,  the immaculate Hecke poset, which is the poset of Theorem~\ref{thm:all4actions}.  Furthermore, the cover relation  $S \prec_{\mathcal{A}^\ast _\alpha} T$ determines exactly the same poset on $\SIT(\alpha)$. The map $\phi$  allows us to identify $\SIT(\alpha/\beta)$ with a subposet of $\SIT(\alpha)$.   
Hence we have the following proposition. 

\begin{proposition}\label{prop:skew-partial-order}  The cover relation defined by 
\[ S\prec_{\rdI _{\alpha/\beta}} T \text{ if and only if } \phi(T)=\pi_{i+m}^{\rdI}(\phi(S)) \text{ for some }i=1,2,\ldots, n-m-1,\] is a partial order on the elements of $\SIT(\alpha/\beta)$.

Moreover, the cover relation  
\[ S\prec_{\mathcal{A}^\ast _{\alpha/\beta}} T \text{ if and only if } \phi(T)=\pi_{i+m}^{\mathcal{A}^\ast}(\phi(S)) \text{ for some }i=1,2,\ldots, n-m-1,\]
determines the same poset structure on $\SIT(\alpha/\beta)$.
\end{proposition}
We call the common poset $\SIT(\alpha/\beta)$ defined by either of the two actions $\pi^{\rdI}$ or $\pi^{\mathcal{A}^*}$  the \emph{skew immaculate Hecke poset} for the skew shape $\alpha/\beta$, and denote it by $P\rdI_{\alpha/\beta}$.

The following lemma is the skew analogue of \cite[Lemma~6.2]{NSvWVW2024}.

\begin{lemma}\label{lem:sameposet-skew} Let $\alpha\vDash n$, $\beta\vDash m\le n$ with $\beta\subseteq \alpha$, and $S,T\in \SIT(\alpha/\beta).$ Then 
\[S\svw{\prec_{\dI _{\alpha/\beta}}}  T \iff S\svw{\prec_{\rdI _{\alpha/\beta}}}  T.\]
\end{lemma}

Using the map $\phi$, one sees that the proof is identical to  \cite[Lemma~6.2]{NSvWVW2024} and is therefore omitted.
\begin{remark}\label{rem:dI-action and barA-action}
It follows then as in \cite{NSvWVW2024} that the two cover relations
 \[ S\prec_{\dI _{\alpha/\beta}} T \text{ if and only if } \phi(S)=\pi_{i+m}^{\dI}(\phi(T)) \text{ for some }i=1,2,\ldots, n-m-1,\]
\[ S\prec_{\mathcal{\bA}^\ast _{\alpha/\beta}} T \text{ if and only if } \phi(S)=\pi_{i+m}^{\mathcal{\bA}^\ast}(\phi(T)) \text{ for some }i=1,2,\ldots, n-m-1,\]
define the same poset structure on $\SIT(\alpha/\beta)$, which we denote by $P\dI_{\alpha/\beta}$. 
\end{remark}

Note that as a consequence of these definitions, which have been chosen to be consistent with the definitions in \cite{NSvWVW2024},  we get the following result.

\begin{proposition}\label{prop:isoposets}
Let $\alpha, \beta$ be compositions with $\beta \subseteq \alpha$. Then the posets  $P\dI_{\alpha/\beta}$ and $P\rdI_{\alpha/\beta}$ are isomorphic.
\end{proposition}

By Proposition~\ref{prop:poset-to-composition-serieS} (1), Theorem~\ref{thm:8flavours-skew-tableaux} and Proposition~\ref{prop:skew-action} we \svw{obtain the next theorem.}

\begin{theorem}\label{the:4skewmodules} Let $\alpha \vDash n, \beta \vDash m$ with $\beta \subseteq \alpha$. Then we have the following.
\begin{itemize}
\item $\mathcal{V}_{\alpha/\beta} = \spam \{ T : T\in \SIT (\alpha / \beta) \}$ is an $H_{n-m}(0)$-module for the $\rdI$-action. Its quasisymmetric characteristic is $\rdI_{\alpha/\beta}$.
\item $\mathcal{W}_{\alpha/\beta} = \spam \{ T : T\in \SIT (\alpha / \beta) \}$ is an $H_{n-m}(0)$-module for the $\dI$-action. Its quasisymmetric characteristic is $\dI_{\alpha/\beta}$.
\item $\mathcal{A}_{\alpha/\beta} = \spam \{ T : T\in \SIT (\alpha / \beta) \}$ is an $H_{n-m}(0)$-module for the $\mathcal{A}^*$-action. Its quasisymmetric characteristic is $\mathcal{A}^*_{\alpha/\beta}$.
\item $\mathcal{\bA}_{\alpha/\beta} = \spam \{ T : T\in \SIT (\alpha / \beta) \}$ is an $H_{n-m}(0)$-module for the $\bA^*$-action. Its quasisymmetric characteristic is $\bA^*_{\alpha/\beta}$.

\end{itemize}
\end{theorem}

We will sometimes refer to such modules as \emph{skew} modules to \svw{emphasise} their reliance on a skew diagram. 

\svw{Let }us now turn our attention to $\SET (\alpha/\beta)$ where all the column entries increase from bottom to top. Observe in Example~\ref{ex:skewSIT-to-straightSIT} that \svw{in fact $T\in \SET(\alpha/\beta)$, and so} we can have $T\in \SET(\alpha/\beta)$ but $\phi(T) \not\in \SET(\alpha)$. However, we do have the following result that gives us two more skew modules.

\begin{proposition}\label{prop:skew-SET-module} Let $\alpha \vDash n, \beta \vDash m$ with $\beta \subseteq \alpha$.

For the $\rdI$-action,  we have that $\spam \{ T : T\in \SET (\alpha / \beta) \}$  is an $H_{n-m}(0)$-submodule of  $\mathcal{V}_{\alpha/\beta}$ with quasisymmetric characteristic $\mathcal{R}\mathcal{E}_{\alpha/\beta}$.  

For the $\mathcal{A}^*$-action,  we have that $\spam \{ T : T\in \SET (\alpha / \beta) \}$  is an $H_{n-m}(0)$-submodule of  $\mathcal{A}_{\alpha/\beta}$ with quasisymmetric characteristic ${\mathcal{A}_{\SET({\alpha/\beta})}^*}$.
\end{proposition}
\begin{proof}
In order to prove that we have a submodule, we must show that $\SET (\alpha/\beta)$ is closed with respect to our chosen 0-Hecke action, namely if $T\in\SET(\alpha/\beta)$ then $\pi_i^a(T)\in\SET(\alpha/\beta)$ for our 0-Hecke action $\pi^a$.

Let $\pi_i^a$ be a generator for either our $\rdI$-action or $\mathcal{A}^*$-action as described in Proposition~\ref{prop:skew-action}. If $\pia (T) = T$ or $0$ the $\pia (T) \in \SET (\alpha/\beta)$ and we are done. 

Hence, it remains to check the case where
$$\pia(T) = s_i(T).$$
By Proposition~\ref{prop:skew-action} this will be when $i\in \Des(T)$, and by Table~\ref{table:All4Imm} this will be precisely when $i+1$ is strictly below $i$ because we are considering the $\rdI$-action or $\mathcal{A}^*$-action.

Since $T\in \SET (\alpha/\beta)$ we are guaranteed that $i$ and $i+1$ will be in different columns because the entries in every column of $T$ increase from bottom to top by definition. Hence in $T$ we will see one of the following.
$$\tableau{&&i\\ \\ \scriptstyle{i+1}}\qquad \textrm{or}\qquad\tableau{i&&\\ \\&&\scriptstyle{i+1}}$$
Now observe that in each of these cases switching $i$ and $i+1$ does not affect the increasing from left to right row condition, nor the increasing from bottom to top column condition. Hence $\pia (T) \in \SET (\alpha/\beta)$ in this final case, and we are done. 
 \end{proof}

\section{Branching rules}\label{sec:Skew-branching}
Throughout this section, unless explicitly mentioned to the contrary, by 0-Hecke action we will mean the $\rdI$-action of $\hn$.  In particular, for ease of comprehension in this section \svw{we will write $\pi_i(T)$ for the $\rdI$-action $\pi^{\rdI}_i(T)$}. 

The goal of this section is to establish a  branching rule for the 0-Hecke module $\mathcal{V}_{\alpha}$ with respect to the $\rdI$-action, using the  modules $\mathcal{V}_{\alpha/\beta}$ established in Section~\ref{sec:Skew-poset-moduleS}.     An analogous branching rule for the 0-Hecke module afforded by standard reverse composition tableaux was established by Tewari and van Willigenburg  \cite[Theorem~9.10]{TvW2015}. 
 Their argument can be adapted to our case, as we now show. 

The following explicit identification of  $H_m(0)\otimes H_{n-m}(0)$ as a subalgebra of $\hn$ is described in \cite[Th\'eor\`eme 1 (iii)]{DKLT1996}, and  used in \cite{TvW2015}.
Define $H_{m, n-m}(0)$ to be the subalgebra of $\hn$ generated by the set 
\[\{\pi_1,\ldots, \pi_{m-1},  \pi_{m+1},\ldots,  \pi_{n-1}\}.\]
 The module $H_m(0)\otimes H_{n-m}(0)$  is generated by the set 
 \[\{\pi_1\otimes 1,\ldots, \pi_{m-1}\otimes 1, 1\otimes \pi_{1},\ldots, 1\otimes \pi_{n-1-m}\},\]
where 1 is the unit of the 0-Hecke algebra.  The identification is made 
 by means of the map 
\begin{equation}\label{eqn:submodule-for-induction}
\pi_i\mapsto \begin{cases} \pi_i\otimes 1,& 1\le i\le m-1, \\
                                            1\otimes \pi_{i-m}, & m+1\le i \le n-1.
                       \end{cases}
\end{equation}

We will prove the following theorem.

\begin{theorem}\label{thm:Skew-branch} Let $\alpha\vDash n$. 
The restriction of the $\hn$-module $\mathcal{V}_{\alpha}$  to the subalgebra $H_m(0)\otimes H_{n-m}(0)$ admits the following decomposition.
\[\mathcal{V}_{\alpha}\big\downarrow^{\hn}_{H_m(0)\,\otimes\, H_{n-m}(0)} = 
\bigoplus_{\substack{\beta\vDash m  \\ \beta\subseteq\alpha}} \,
\mathcal{V}_\beta \otimes \mathcal{V}_{\alpha/\beta}\] 
In particular, 
\[\mathcal{V}_\alpha\big\downarrow^{\hn}_{H_{n-1}(0)} \cong \bigoplus_{\substack{\beta\vDash n-1  \\ \beta\subset\alpha}} \mathcal{V}_\beta.  \]
That is, the sum on the right runs over all compositions $\beta $ such that $\alpha/\beta$  consists of a single cell.
\end{theorem}
However, before we prove Theorem~\ref{thm:Skew-branch} we require some other results that we now give. From 
 Theorem~\ref{thm:two-varS}, we can conclude  the following proposition, which  already reflects the decomposition claimed above,  at the level of dimensions. 
\begin{proposition}\label{prop:skew-SIT-decomp} For a skew diagram $\alpha/\beta$ and a set $S$ of positive integers of size $|\alpha/\beta|$, let $\SIT_S(\alpha/\beta)$ denote the set of standard immaculate tableaux of shape $\alpha/\beta$ with distinct entries in the set $S$. In particular, for a composition $\gamma \vDash m$ we have that $\SIT(\gamma)=\SIT_{\{1,2,\ldots,m\}}(\gamma)$.
Then we have, for fixed $\alpha\vDash n$, 
\[\SIT(\alpha)=\bigcup_{m\ge 1} \bigsqcup_{\substack {\beta\vDash m\\ \beta\subseteq\alpha}} \SIT(\beta)\times \SIT_{\{|\beta|+1,|\beta|+2,\ldots, |\alpha|\}}(\alpha/\beta).\]
\end{proposition}

 For $T\in \SIT(\alpha)$ let $T_{\le m}$ be the subtableau of $T$ formed by the entries that are at most $m$, and let $T_{> m}$ be the subtableau of $T$ formed by the entries  greater than $m$. Similar to \cite{TvW2015}, for $\alpha \vDash n$ and $\beta\vDash m $ such that $\beta\subseteq \alpha$, define
\[X_{\alpha,\beta}=\{T\in \SIT(\alpha):  \text{ shape of } T_{>m} \text{ is } \alpha/\beta\}.\]

Then Proposition~\ref{prop:skew-SIT-decomp} may be rephrased as 
\begin{equation}\label{eqn:skew-SIT-decomp} \SIT(\alpha)=\bigcup_{m\ge 1}\bigsqcup_{\substack {\beta\vDash m\\ \beta\subseteq\alpha} }X_{\alpha,\beta} .
    \end{equation}

The overarching goal of the arguments in \cite[Section 9, pages 1061-1063]{TvW2015} is to make the precise formal identifications that are \svw{required}  to lift \eqref{eqn:skew-SIT-decomp} to the module level, which we now follow.

 Define the following vector spaces:
 \begin{itemize}
 \item  $\mathbf{S}_{\beta}$ is the $\mathbb{C}$-linear span of all tableaux in $\SIT(\beta)$;
 \item $\mathbf{S}_{\alpha/\beta}$ is the $\mathbb{C}$-linear span of all tableaux in $\SIT(\alpha/\beta)$;
     \item $\mathbf{S}_{X_{\alpha,\beta}}$ is the $\mathbb{C}$-linear span of all tableaux in $X_{\alpha,\beta}$.
 \end{itemize}
 There is a natural vector space isomorphism $\theta: \mathbf{S}_{X_{\alpha,\beta}}\rightarrow  
 \mathbf{S}_{\beta}\otimes \mathbf{S}_{\alpha/\beta}$
 given by 
 \begin{equation}\label{eqn:Straight-to-skew-iso}
 \theta(T)= T_{\le m}\otimes T_{>m}.
 \end{equation}

   Recalling that $H_m(0)\otimes H_{n-m}(0)$ is generated by 
the set $\{\pi_i\otimes 1, 1\otimes \pi_{j}: 1\le i\le m-1, 1\le j\le n-m-1\}$, we see that $\mathbf{S}_{X_{\alpha,\beta}}$ becomes an $H_m(0)\otimes H_{n-m}(0)$-module by setting 
\begin{alignat*}{3} (\pi_i\otimes 1)\cdot T&= \pi_i(T), &&\quad\text{\ if $1\le i\le m-1$,}\\
                (1\otimes \pi_{j}) \cdot T&= \pi_{m+j}(T), &&\quad\text{\ if $1\le j\le n-m-1$}. 
\end{alignat*}

Equivalently,  this definition equips $\mathbf{S}_{X_{\alpha,\beta}}$ with the structure of an $H_{m, n-m}(0)$-module, where $\pi_i \cdot T= \pi_i(T)$ for $1\le i\le m-1$, $m+1\le i\le n-1$.  Moreover it is straightforward to verify that the $H_{m, n-m}(0)$-action commutes with the isomorphism $\theta$ described in \eqref{eqn:Straight-to-skew-iso}.
We now immediately have the following analogue of \cite[Proposition 9.9]{TvW2015}.
\begin{proposition}\label{Hecke-Subgp-action} For $\alpha \vDash n, \beta\vDash m$ and $\beta \subseteq \alpha$ the map $\theta: T\mapsto  T_{\le m}\otimes T_{>m}$ defines an isomorphism of $H_m(0)\otimes H_{n-m}(0)$-modules for the $\rdI$-action 
\[\mathbf{S}_{X_{\alpha,\beta}}\cong 
 \mathbf{S}_{\beta}\otimes \mathbf{S}_{\alpha/\beta}.\]
    \end{proposition}

\begin{proof}[Proof of Theorem~\ref{thm:Skew-branch}]  Considering the above discussions, the proof is now  immediate upon observing that 
\begin{itemize}
\item
$\mathbf{S}_\alpha= \mathcal{V}_\alpha$ as $H_n(0)$-modules, and 
\item there is an $H_m(0)\otimes H_{n-m}(0)$-module isomorphism 
\[\mathbf{S}_\alpha\cong \bigoplus_{\substack{\beta\vDash m  \\ \beta\subseteq\alpha}} \,
\mathbf{S}_{X_{\alpha,\beta}}. \]
\end{itemize}
The special case $m=n-1$ is a consequence of the fact that $H_1(0)\cong \mathbb{C}$.
\end{proof}

 We conclude this section with the observation that Theorem~\ref{thm:Skew-branch} 
 holds if the module $\mathcal{V}_\alpha$ carrying the $\rdI$-action is replaced with the modules for the $\dI$-action, or the $\mathcal{A}^*$- and $\bA^*$-actions. 
  The precise statements appear below.  We omit the proofs, since they follow almost identically from the above arguments for Theorem~\ref{thm:Skew-branch}.

 \begin{theorem}\label{thm:Skew-branch-all4flavours} Let $\alpha\vDash n$, and  let $\mathcal{U}_\alpha$ denote the $\hn$-module defined by any one of the four actions  on the  set $\SIT(\alpha)$ arising from  the four descent sets $\Des_{\dI}$, $\Des_{\rdI}$, $\Des_{\mathcal{A}^*}$, $\Des_{\bA^*}$.  
 Then for fixed $m\le n$ we have 
\[\mathcal{U}_{\alpha}\big\downarrow^{\hn}_{H_m(0)\,\otimes\, H_{n-m}(0)} = 
\bigoplus_{\substack{\beta\vDash m  \\ \beta\subseteq\alpha}} \,
\mathcal{U}_\beta \otimes \mathcal{U}_{\alpha/\beta}.
\quad\text{ In particular, }\quad
\mathcal{U}_\alpha\big\downarrow^{\hn}_{H_{n-1}(0)} \cong \bigoplus_{\substack{\beta\vDash n-1  \\ \beta\subset\alpha}} \mathcal{U}_\beta . \]
\end{theorem}

The analogous result holds for each of these four actions on $\SET(\alpha),$ the subset of $\SIT(\alpha)$ in which all columns increase. We defer this discussion to Section~\ref{sec:Skew-Hecke-poset}.

\section{The skew immaculate Hecke poset}\label{sec:Skew-Hecke-poset}
Let $\alpha\vDash n$, $\beta\vDash m$ with $\beta\subseteq \alpha$.  In this section we show how the results of \cite{NSvWVW2024} on the immaculate Hecke poset $P\rdI_\alpha$ generalise to the skew immaculate poset $P\rdI_{\alpha/\beta}$.

We begin with an enumerative formula extending  a result of \cite{BBSSZ2014}.

\begin{proposition} Let $\alpha, \beta$ be compositions with $\beta \subseteq \alpha$, and $\ell=\ell(\alpha)$, $k = \ell(\beta)$. Let 
    $\gamma = (\alpha_{k+1},\ldots,\alpha_{\ell})\vDash m$.  Then 
    \begin{enumerate}
        \item 
        $|\SIT(\gamma)|=\frac{m!}
        {m(m-\gamma_1) (m-(\gamma_1+\gamma_2))\cdots 
        (m-\sum_{j=1}^{\ell-k} \gamma_j) \prod_{i=1}^{\ell(\gamma)}(\gamma_i-1)!}$
        \item $|\SIT(\alpha/\beta)|=|\SIT(\gamma)|\cdot \binom{n}{|\gamma|,\, \alpha_k-\beta_k,\, \ldots,\, \alpha_1-\beta_1}.$
    \end{enumerate}
\end{proposition}
\begin{proof} Item (1) is the result in \cite[Proposition~3.13]{BBSSZ2014}.

Item (2) follows because the diagram of $\gamma$ can be filled independently of the skew diagram  consisting of the rows below it. 
\end{proof}

In particular,  when $\ell(\beta)=\ell(\alpha)=k$, the composition $\gamma$ is empty and the rows of the skew diagram can be filled independently of one  other.  Because of the row increase condition, in this case the size of $\SIT(\alpha/\beta)$ is simply the multinomial coefficient $\binom{n}{\alpha_k-\beta_k, \ldots, \alpha_1-\beta_1}.$

We now define some tableaux that will play a key role in our analysis of $P\rdI_{\alpha/\beta}$. When $\beta = \emptyset$ these definitions coincide with their counterparts in Definition~\ref{def:maxminels}.

\begin{definition}\label{def:skewels} For $\alpha \vDash n, \beta\vDash m$ with $\beta\subseteq\alpha$, we define the following special standard tableaux in $\SIT(\alpha/\beta)$. 
\begin{itemize}
\item $\svw{S}^0_{\alpha/\beta}$ is the  tableau in which the cells of the first column of $\alpha$, if any remain in $\alpha/\beta$ (i.e., if $\ell(\alpha)-\ell(\beta)\ne 0$), are filled  with 
$1,\ldots, \ell(\alpha)-\ell(\beta)$, increasing bottom to top;  then fill the remaining cells by rows, \emph{top to bottom}, left to right with consecutive integers starting at $\ell(\alpha)-\ell(\beta)+1$ and ending at $n-m.$  
\item $S^{row}_{\alpha/\beta}$ is the \emph{row superstandard}  tableau  whose entries left to right, \emph{bottom to top},  beginning with the bottom row and moving up, using the numbers $1,2, \ldots, n-m$ taken in consecutive order.  
\item $S^{col}_{\alpha/\beta}$ is the \emph{column superstandard}   tableau  whose columns are  filled   bottom to top,  left to right, beginning with the leftmost column and moving   right, using the numbers $1,2,\ldots, n-m$  taken in consecutive order.  
 \end{itemize}
\end{definition}

\begin{example}\label{ex:special-skew-SIT-1}  Let $\alpha=(2,2,3,2,4)  \vDash 13$, $\beta=(2,1,2)\vDash 5$.  
\[\svw{S}^0_{\alpha/\beta}\!=\tableau{2& 3&4 &5\\ 1 & 6\\ \bullet &\bullet &7\\ \bullet  &8\\\bullet &\bullet}\!
S^{row}_{\alpha/\beta}\!= \tableau{5& 6&7 &8\\ 3 & 4\\\bullet &\bullet&2\\ \bullet &1\\\bullet &\bullet}\!
S^{col}_{\alpha/\beta}\!=\tableau{2& 5&7 &8\\ 1 & 4\\ \bullet &\bullet&6\\ \bullet  &3\\\bullet &\bullet}
\]

\end{example}

\begin{example}\label{ex:special-skew-SIT-2}  Let $\alpha=(5,4,6)\vDash 15$, $\beta=(2,1,2)\vDash 5$. Note that here $\ell(\beta)=\ell(\alpha)=3$.  
\[\svw{S}^0_{\alpha/\beta} =
\tableau{\bullet & \bullet &1 &2 &3 &4 \\ \bullet   &5 & 6 &7\\ \bullet &\bullet &8 &9 &10}\quad 
S^{row}_{\alpha/\beta}= \tableau{\bullet &\bullet &7 &8 &9 &10 \\ \bullet   & 4&5 & 6 \\ \bullet &\bullet &1 &2 &3}\quad 
S^{col}_{\alpha/\beta}=\tableau{\bullet & \bullet &4 &7 &9 &10 \\ \bullet  &1 & 3 &6\\ \bullet &\bullet &2 &5 &8}
\]

\end{example}

Next we examine the module structures afforded by the skew Hecke poset $P\rdI_{\alpha/\beta}$ more carefully.
We establish the analogues of \cite[Proposition~6.6, Proposition 6.11]{NSvWVW2024}. We begin by  focusing on the row-strict immaculate action $\pi^{\rdI}$. Recall that $s_i$ is the operator \svw{that} switches the entries $i$, $i+1$ in a tableau $T$. The $\pi^{\rdI}$-action of the  Hecke generator $\pi_i ^{\rdI}$  on  $T\in \SIT(\alpha/\beta)$ is given precisely as follows by \Cref{prop:skew-action} and \Cref{table:All4Imm}. As in the previous section we will denote $\pi_i ^{\rdI}$ by $\pi_i$ for ease of comprehension.
\[\pi _i(T)= \pi^{\rdI}_i(T)=\begin{cases} T, &  \text{if $i+1$ is strictly above $i$ in $T$},\\
                        s_i(T), & \text{if $i+1$ is strictly below $i$ in $T$},\\
                        0, & \text{otherwise}.
 \end{cases}\]

The straightening algorithms below  will allow us to conclude, using \Cref{prop:poset-to-composition-serieS}, that the row-strict skew immaculate module $\mathcal{V}_{\alpha/\beta}$ is cyclically generated by $S^0_{\alpha/\beta}$, and the skew immaculate module $\mathcal{W}_{\alpha/\beta}$ is cyclically generated by $S^{row}_{\alpha/\beta}$.

We begin with a lemma. See Example~\ref{ex:skew-bot-to-T} for an illustration.
\begin{lemma}\label{lem:consecutive-gens}  Let $1\le b\le a\le n-1$. Let $\alpha\vDash n$ and $\beta\subseteq \alpha$. Suppose $S, T\in \SIT(\alpha/\beta)$, $S\ne T$,  and suppose 
$\pi_b, \pi_{b+1},\ldots,\pi_a$ is a sequence of generators such that 
\begin{itemize} 
\item $T=\pi _a \pi _{a-1}\cdots \pi _{b+1}\pi _b(S)$. 
\item For each $i=b, b+1, \ldots, a$, $T_i=\pi_i \pi_{i-1}\cdots \pi_{b+1}\pi_b(S) \in \SIT(\alpha/\beta)$ and $T_i\ne T_{i-1}$, where we set $T_{b-1}=S$.
\end{itemize}
\svw{Then $T$ is obtained from $S$ by replacing the entries   $b, b+1, b+2, \ldots , a, a+1$ by the entries $a+1, b, b+1, \ldots , a-1, a$, respectively.}
\end{lemma}
\begin{proof}  The hypotheses guarantee that at each step of applying the sequence of operators $\pi_a \pi _{a-1}\cdots \pi _{b+1}\pi _{b}$,  the resulting tableau $T_i$ is in $\SIT(\alpha/\beta)$, and is the result of swapping $i$ and $i+1$ in $T_{i-1}$; in particular the result  is always nonzero.%

The statement now follows immediately.
    \end{proof}
\begin{proposition}\label{prop:bot-elt-skew} Let $\alpha\vDash n$, $\beta\vDash m$ with $\beta\subseteq \alpha$.
Consider the skew standard immaculate tableau $S^0_{\alpha/\beta}$.  Then for any $T\in \SIT({\alpha/\beta})$ where $T\ne S^0_{\alpha/\beta},$ there is a sequence of generators $\pi_{j_i}, $ and distinct tableaux $T_i\in \SIT({\alpha/\beta}),$ $i=1,\ldots, r,$  such that $\pi_{j_i}(T_{i})=T_{i-1}, i=1,2,\ldots ,r,$ where we set $T_{0}=T$ and $T_r=S^0_{\alpha/\beta}.$ Hence we conclude 
\[ S^0_{\alpha/\beta} \poRI T \text{ and } T=\pi_{j_1}\pi_{j_{2}}\cdots\pi_{j_r}(S^0_{\alpha/\beta}).\]
In particular, \svw{$S^0_{\alpha/\beta}$} is the unique minimal element of the poset $P\rdI_{\alpha/\beta}.$
\end{proposition}
Before we prove this, we illustrate how to construct a saturated chain from  $S^0_{\alpha/\beta}$ to any tableau $T$  in $\SIT({\alpha/\beta})$ by applying 0-Hecke operators for the $\rdI$-action. \svw{Lemma~\ref{lem:consecutive-gens}} is helpful in compressing the steps involving consecutive generators.
\begin{example}\label{ex:skew-bot-to-T}  Let $\alpha=(4,3,4,2,3)$, $\beta=(2,1,2)$, and
$$T=\tableau{5&8 &10\\3 &4\\\bullet &\bullet &6 &11\\ \bullet &1 &9\\
\bullet &\bullet &2 &7}\qquad S^0_{\alpha/\beta}=\tableau{2&3 &4\\1 &5\\\bullet &\bullet &6 &7\\ \bullet &8 &9\\
\bullet &\bullet  &10 &11}.$$
We work first to match column 1 of $T$ with column 1 of $S^0_{\alpha/\beta}$.  Using Lemma~\ref{lem:consecutive-gens}, we have,  starting with the lowest entry in column 1 that differs from $S^0_{\alpha/\beta}$, 
\[T\xleftarrow[]{\pi_2\pi_1}
\tableau{5&8 &10\\1 &4\\\bullet &\bullet &6 &11\\ \bullet &2 &9\\
\bullet &\bullet &3 &7}=T_1
\xleftarrow[]{\pi_4\pi_3\pi_2}
\tableau{2&8 &10\\1 &5\\\bullet &\bullet &6 &11\\ \bullet &3 &9\\
\bullet &\bullet &4 &7}=T_2.\]
To be clear about the sequence of generators, this means $T=\pi_2\pi_1(T_1)$ and $T_1=\pi_4\pi_3\pi_2(T_2)$.
Now we work downwards from the topmost row of length at least two, one entry at a time, starting with the smallest entry that does not match the one in  $S^0_{\alpha/\beta}$, to make this row agree with that of $S^0_{\alpha/\beta}$.  The entries that need to be replaced are first 8 and then 10.
\[T_2
\xleftarrow[]{\pi_7\pi_6\pi_5\pi_4\pi_3}
\tableau{2&3 &10\\1 &6\\\bullet &\bullet &7 &11\\ \bullet &4 &9\\
\bullet &\bullet &5 &8}=T_3\\
\xleftarrow[]{\pi_9\pi_8\pi_7\pi_6\pi_5\pi_4} 
\tableau{2&3 &4\\1 &7\\\bullet &\bullet &8 &11\\ \bullet &5 &10\\
\bullet &\bullet &6 &9}=T_4.
\]
That is, $T_2= \pi_7\pi_6\pi_5\pi_4\pi_3(T_3) $ and $T_3=\pi_9\pi_8\pi_7\pi_6\pi_5\pi_4(T_4)$.\\
Proceeding to the next highest row differing from $S^0_{\alpha/\beta}$, we have
\[T_4\xleftarrow[]{\pi_6\pi_5}
\tableau{2&3 &4\\1 &5\\\bullet &\bullet &8 &11\\ \bullet &6 &10\\
\bullet &\bullet &7 &9}=T_5
\xleftarrow[]{\pi_7\pi_6}
\tableau{2&3 &4\\1 &5\\\bullet &\bullet &6 &11\\ \bullet &7 &10\\
\bullet &\bullet &8 &9}=T_6,  \]
and thus $T_4=\pi_6\pi_5(T_5)$, $T_5=\pi_7\pi_6(T_6)$.
Finally we have 
\[T_6
\xleftarrow[]{\pi_{10}\pi_9\pi_8\pi_7} \tableau{2&3 &4\\1 &5\\\bullet &\bullet &6 &7\\ \bullet &8 &11\\
\bullet &\bullet  &9 &10}
= T_7 \xleftarrow[]{\pi_{10}\pi_9}
\tableau{2&3 &4\\1 &5\\\bullet &\bullet &6 &7\\ \bullet &8 &9\\
\bullet &\bullet  &10 &11}=S^0_{\alpha/\beta}.  
\]

Again, this means  
$T_6=\pi_{10}\pi_9\pi_8\pi_7(T_7)$ and $T_7 = \pi_{10}\pi_9(S^0_{\alpha/\beta})$. 

The final result is the following saturated chain from $S^0_{\alpha/\beta}$ to $T$.
\[T=(\pi_2\pi_1)\, (\pi_4\pi_3\pi_2)\, (\pi_7\pi_6\cdots\pi_3)\, (\pi_9\pi_8\cdots \pi_4)\, (\pi_6\pi_5)\, (\pi_7\pi_6)\, (\pi_{10}\pi_9\pi_8\pi_7)\, (\pi_{10}\pi_9)(S^0_{\alpha/\beta} )\]
\end{example}

The algorithm executed in the preceding \svw{example} identifies a unique saturated  chain from $S^0_{\alpha/\beta}$ to $T$, and is identical to the one in the proof of \cite[Proposition~6.6]{NSvWVW2024}.  Readers familiar with that argument may therefore safely skip the following proof of the validity of the algorithm, given here for completeness.

\begin{proof}[Proof of Proposition~\ref{prop:bot-elt-skew}] 
In what follows, when an integer  $a$ occupies row $p$ and column $q$ of $T,$ by the counterpart of $a$ in $S^0_{\alpha/\beta}$ we will mean the integer occupying the same cell, row $p$ and column $q$, in $S^0_{\alpha/\beta}.$ 
First assume $\ell(\alpha)\ne \ell(\beta)$. This means there are cells in the first column of $\alpha$ appearing in the first column of $\alpha/\beta$.  
\begin{enumerate}
    \item[Step 1:] We begin by making the first column of $T$ match the first column of $S^0_{\alpha/\beta}.$  Here it is important that some cells in the first column of  $\alpha$ remain in $\alpha/\beta$, so that $\ell(\alpha)-\ell(\beta)\ge 1$. Find the least $j$, $1\le j\le \ell(\alpha)-\ell(\beta),$ such that the entry $x$ in cell $(j+\ell(\beta),1)$ \svw{of $T$} is not equal to $j.$ 

    Then in $T$ we have that $x-1$ is in a lower row, not in column 1 by minimality of $j$, and the fact that rows increase left to right, and the first column increases. Hence $T=\pi_{x-1}(T_1),$ such that in $T_1\in\SIT({\alpha/\beta}),$  $x-1$ is now a descent, lying in a row  strictly \textit{higher} than $x.$  Now repeat this procedure until $x$ is replaced by $j.$ Then continue with the next entry in column 1 of $T$ \svw{that} does not match 
    in $S^0_{\alpha/\beta}.$  Clearly this process ends with a tableau $T_r=\pi_{x_r}\pi_{x_{r-1}}\cdots \pi_{x_1}(T),$ whose first column matches  column 1 of $S^0_{\alpha/\beta}.$
    
    Note that   if $T$ and $S^0_{\alpha/\beta}$ already agree in the first column \svw{then $T=T_r$,} so Step 1 is not necessary.     
    \item[Step 2:] First observe that $T_r$ and $S^0_{\alpha/\beta}$ now agree for all entries less than or equal to the positive number $\ell=\ell(\alpha)-\ell(\beta).$ Now consider the topmost row of length greater than 1, say row $k$.  Find the least entry, say $y,$ in this row of $T_r$ that differs from the corresponding entry in $S^0_{\alpha/\beta}.$ Note that $y$ is then necessarily larger than its counterpart in $S^0_{\alpha/\beta},$ by definition of the latter. Then \svw{$y> \ell+1$ because if $y=\ell +1$ then it would already be in the correct place by the definition of $S^0_{\alpha/\beta}$. Hence we are guaranteed that}  $y-1$ is strictly below $y$ in $T_r$ because we are in the topmost row possible, and rows increase left to right.  
    Hence 
    $T_r=\pi_{y-1}(T_{r+1})$ for $T_{r+1}\in \SIT({\alpha/\beta}),$ such that $y-1$ is a descent in $T_{r+1}$  strictly \textit{higher} than $y$.  We repeat this step until $y$ has been replaced by its counterpart in $S^0_{\alpha/\beta}.$  
    
    \item[Step 3:] Continue in this manner to the end of the row. We now have a sequence of operators $\pi_{i_j}$ and tableaux $T_j\in \SIT({\alpha/\beta})$ such that $T_{j-1}=\pi_{i_j}(T_j)$, and the final tableau $T_s$ agrees with $S^0_{\alpha/\beta}$ for all entries $\le \ell+(\alpha_k-1)$  if $\beta_k = 0$, or $\le \ell+(\alpha_k - \beta_k)$ if $\beta _k \geq 1$.
    
    \item[Step 4:] Proceed  downwards to the next row where an entry in $T_s$ differs from its counterpart in $S^0_{\alpha/\beta},$ and repeat Steps 2 and 3, until all rows \svw{match $S^0_{\alpha/\beta}$.} 
\end{enumerate}

Since at the end of each iteration of Step 3,  the number of entries that are in agreement  with $S^0_{\alpha/\beta}$ increases, we see that the algorithm produces a saturated chain from $S^0_{\alpha/\beta}$ to $T$ in the poset $P\rdI_{\alpha/\beta}$ as claimed. \svw{This completes the case $\ell(\alpha)\neq\ell(\beta)$.}

\svw{The case $\ell(\alpha)=\ell(\beta)$  is almost identical, with the simplification that we proceed directly to Step 2, since Step 1 is no longer needed. }\end{proof}

\svw{We now immediately conclude the following, using \Cref{prop:poset-to-composition-serieS} (2).}  The statement about the $\mathcal{A}^*$-action follows since, by \Cref{prop:skew-partial-order}, the $\rdI$- and $\mathcal{A}^*$-actions both determine the same  poset structure $P\rdI_{\alpha/\beta}$ on $\SIT(\alpha/\beta)$. 

\begin{theorem}\label{thm-rdI-skew} 
Let $\alpha\vDash n$, $\beta\vDash m$ with $\beta \subseteq \alpha$. Then $\mathcal{V}_{\alpha/\beta}$
is a \svw{cyclic}
$H_{n-m}(0)$-module \svw{generated} by the element 
$ S^0_{\alpha/\beta}$. 
The same statement holds for the module $\mathcal{A}_{\alpha/\beta}$ defined by the $\mathcal{A}^*$-action on $\SIT(\alpha/\beta)$. 
\end{theorem}

\begin{proposition}\label{prop:top-elt-skew} 
Let $\alpha\vDash n$, $\beta \vDash m$ with $\beta \subseteq \alpha$. Consider the skew standard immaculate tableau $ S^{row}_{\alpha/\beta}$.  Then for any $T\in \SIT({\alpha/\beta})$ where $T\ne S^{row}_{\alpha/\beta},$ there is a sequence of generators $\pi_{j_i}, $ and distinct tableaux $T_i\in \SIT({\alpha/\beta}),$ $i=1,\ldots, r,$ such that $\pi_{j_i}(T_{i-1})=T_{i}, i=1,2,\ldots ,r,$ where we set $T_{0}=T$ and 
$T_r=S^{row}_{\alpha/\beta}$. Hence we conclude 
\[ T \poRI S^{row}_{\alpha/\beta} \text{ and } \pi_{j_r}\pi_{j_{r-1}}\cdots\pi_{j_1}(T)=S^{row}_{\alpha/\beta}.\]
In particular, $ S^{row}_{\alpha/\beta}$ is the unique maximal element of the poset $P\rdI_{\alpha/\beta}.$
\end{proposition}

Again we give an example illustrating the straightening algorithm that produces a saturated chain of 0-Hecke operators for the $\rdI$-action, going from a tableau $T\in \SIT({\alpha/\beta})$ 
to the tableau $ S^{row}_{\alpha/\beta}$.

\begin{example}\label{ex:skew-T-to-top}
Let $\alpha=(4,3,4,2,3)$, $\beta=(2,1,2,1)$, and $$T=\tableau{4&7&9\\\bullet &3\\\bullet &\bullet &5&10\\ \bullet &1 &8\\
\bullet &\bullet &2 &6} \qquad S^{row}_{\alpha/\beta}= \tableau{8&9 &10\\  \bullet &7 \\  \bullet &\bullet &5 &6\\ \bullet &3 &4\\
\bullet &\bullet &1 &2}.$$

Start with the largest entry in the topmost row of $T$ that differs from $S^{row}_{\alpha/\beta}$, in this case 9, and apply $\pi_9 $ to $T$. This gives 
\[T=\tableau{4&7&9\\ \bullet &3\\\bullet &\bullet &5&10\\ \bullet &1 &8\\
\bullet &\bullet &2 &6} 
\xrightarrow[]{\pi_9}
\tableau{4&7 &10\\  \bullet &3 \\  \bullet &\bullet &5 &9\\ \bullet &1 &8\\
\bullet &\bullet &2 &6}=T_1.\]
Repeat this step until the topmost row agrees with $S^{row}_{\alpha/\beta}$.  This means we compute $\pi_8\pi_7(T_1)=T_2$ and then $\pi_7\pi_6\pi_5\pi_4(T_2)=T_3$.  Using Lemma~\ref{lem:consecutive-gens}, we have 
\[T_1\xrightarrow[]{\pi_8\pi_7}
\tableau{4&9 &10\\  \bullet &3 \\  \bullet &\bullet &5 &8\\ \bullet &1 &7\\
\bullet &\bullet &2 &6}=T_2
\xrightarrow[]{\pi_7\pi_6\pi_5\pi_4}
\tableau{8&9 &10\\  \bullet &3 \\  \bullet &\bullet &4 &7\\ \bullet &1 &6\\
\bullet &\bullet &2 &5}
=T_3.
\]
Now move to the \svw{second row from} the top.  To make the 3 \svw{in $T_3$} match the 7 in $S^{row}_{\alpha/\beta}$, we must compute $T_4=\pi_6\pi_5\pi_4\pi_3(T_3)$, giving, by Lemma~\ref{lem:consecutive-gens}, 
\[T_3\xrightarrow[]{\pi_6\pi_5\pi_4\pi_3}
\tableau{8&9 &10\\  \bullet &7 \\  \bullet &\bullet &3 &6\\ \bullet &1 &5\\
\bullet &\bullet &2 &4}
=T_4.
\]
Continuing in this way $T_5=\pi_4\pi_3(T_4)$ and \svw{finally} $ S^{row}_{\alpha/\beta}=\pi_2\pi_1(T_5) $, that is,
\[T_4 
\xrightarrow[]{\pi_4\pi_3}
\tableau{8&9 &10\\  \bullet &7 \\  \bullet &\bullet &5 &6\\ \bullet &1 &4\\
\bullet &\bullet &2 &3}=T_5
\xrightarrow[]{\pi_2\pi_1}
\tableau{8&9 &10\\  \bullet &7 \\  \bullet &\bullet &5 &6\\ \bullet &3 &4\\
\bullet &\bullet &1 &2}=S^{row}_{\alpha/\beta}.
\]
The full chain from $T$ to $S^{row}_{\alpha/\beta}$ is then
\[S^{row}_{\alpha/\beta}=(\pi_2\pi_1  )( \pi_4\pi_3 ) (\pi_6\pi_5\pi_4\pi_3 )  ( \pi_7\pi_6\pi_5\pi_4)   ( \pi_8\pi_7) \pi_9(T).\]
\end{example}

{In this case the algorithm, identical to the one in the proof of \cite[Proposition~6.11]{NSvWVW2024}, is simpler to describe.} Again, readers familiar with that proof may safely skip the argument below.
\begin{proof}[Proof of Proposition~\ref{prop:top-elt-skew}] 
 \svw{We} start with the topmost row of $T$ that differs from $ S^{row}_{\alpha/\beta},$ and find the \textit{largest} entry, say $x$, in that row that differs from its counterpart in $ S^{row}_{\alpha/\beta}.$ \svw{Then necessarily $x+1$ is in a lower row of $T=T_0,$ because $x$ is the largest and topmost entry that differs from its counterpart in $ S^{row}_{\alpha/\beta}.$}  Hence $T_1=\pi_x(T_0)$, with $T_1\in \SIT({\alpha/\beta})$ the tableau obtained from $T$ by switching $x, x+1.$  We continue in this manner along each row, top to bottom, right to left.  Note that at the start of a new row $j,$ the tableau $T_k$ agrees with $ S^{row}_{\alpha/\beta}$ \svw{in the top $j-1$ rows.}

It is straightforward to check that this algorithm terminates in the tableau $ S^{row}_{\alpha/\beta}$, producing a saturated chain in the poset from $T$ to $ S^{row}_{\alpha/\beta}.$
\end{proof}

Recall from \Cref{table:All4Imm} and \Cref{prop:skew-action}  that the $\pi^{\dI}$-action is defined on $\SIT(\alpha/\beta)$ explicitly by 
\[\pi^{\dI}_i(T)=\begin{cases} T, &  \text{if $i+1$ is weakly  below $i$ in $T$},\\
                        s_i(T), & \text{if $i+1$ is strictly above $i$ in $T$, and $i, i+1$ are not in column 1 of $\alpha$},\\
                        0, & \text{otherwise}.
 \end{cases}\]
  From Lemma~\ref{lem:sameposet-skew} and Proposition~\ref{prop:top-elt-skew}, using \Cref{prop:poset-to-composition-serieS} \svw{(2)}, we immediately have the following.  The statement about the $\bA^*$-action follows since, by  \Cref{rem:dI-action and barA-action}, the $\dI$- and $\bA^*$-actions both determine the same  poset structure $P\dI_{\alpha/\beta}$ on $\SIT(\alpha/\beta)$.
\begin{theorem}\label{thm-dI-skew} Let $\alpha \vDash n, \beta \vDash m$ with $\beta \subseteq \alpha$. Then $\mathcal{W}_{\alpha/\beta}$ is a \svw{cyclic} $H_{n-m}(0)$-module \svw{generated} by the  \svw{element} $S^{row}_{\alpha/\beta}$. The same statement holds for the module $\bA_{\alpha/\beta}$ defined by the $\bA^*$-action on $\SIT(\alpha/\beta)$ .  
\end{theorem}

Recall that \svw{the \emph{inversion set} of a permutation $\sigma \in S_n$} is
$$\mathrm{Inv} (\sigma) = \{ (p,q)\suchthat 1\leq p < q \leq n \mbox{ and } \sigma(p) > \sigma (q)\}.$$The \emph{number of inversions} is denoted by $\ninv(\sigma) = |\mathrm{Inv}(\sigma)|.$  

\begin{example}\label{ex:inv} 
If $\svw{\sigma = 5\ 4\ 2\ 3\ 1 \in S_5}$, then $\ninv(\sigma) = 9$ from 
$$\mathrm{Inv}(\sigma) = \{(1,2), (1,3), (1,4), (1,5),  (2,3), (2,4), (2,5), (3,5), (4,5)\}.$$
\end{example} 

The \emph{reading word} $\rw(T)$ of a tableau $T$ is defined to be the entries of $T$ read from right to left along rows, and from top to bottom. 
With this in mind, we therefore define the number of inversions of a standard tableau $T $ to be the number of inversions in its reading word $\rw(T)$, and simply write $\inv(T)$ for $\inv \rw(T)$. 

In \cite[Sections 5 and 6]{NSvWVW2024} it was shown that the immaculate Hecke poset $P\rdI_\alpha$ is ranked, with the rank of a standard immaculate tableau $T\in \SIT(\alpha)$  given by the number of inversions in the {reading word} $\rw(T)$ of $T$. The map $\phi$ allows us to conclude that the skew immaculate Hecke poset $P\rdI_{\alpha/\beta}$ is also ranked, with the same rank function.  

Observe that the specification of $\phi$ in Definition~\ref{def:map-skew-to-straight} depends on the choice of a fixed filling of the diagram of $\beta $. Indeed, every $U\in\SIT(\beta)$ determines a unique map $\phi_U:\SIT(\alpha/\beta)\rightarrow \SIT(\alpha)$ for which the statement of  
 Proposition~\ref{prop:skew-action} holds. 

For $\alpha \vDash n, \beta \vDash m$ with $\beta \subseteq \alpha$, let $T\in \SIT(\alpha/\beta)$, and $U\in \SIT(\beta)$. The reading word $\rw $ of $\phi_U(T)$ in $\SIT(\alpha)$ contains inversions of the form $(i,j)$ where $i>m$ and $j\le m$.  These inversions are clearly independent of the choice of $T$ and $U$. We denote their number by $\inv(\alpha, \beta)$,
and examples of such inversions can be seen in Example~\ref{ex:Skew-map-phi}.

\begin{example}\label{ex:Skew-map-phi}  Recall  Example~\ref{ex:skewSIT-to-straightSIT}, with $U\in \SIT(\beta)$ in bold.
\[\text{Let } T=\tableau{3& 5&7 &8\\ 2 & 6\\ \bullet &\bullet &1\\ \bullet  &4\\\bullet &\bullet}\in \SIT(\alpha/\beta).\quad  \text{Then } \phi(T)=\tableau{8& 10 & 12 & 13\\7 & 11\\ {\bm 4} &{\bm 5} &6\\{\bm 3} & 9\\{\bm 1}&{\bm 2}} \in \SIT (\alpha).\]

We have that  $\phi(T)=\phi_U(T)$ for $U=S^{row}_\beta$. Then 
\[\rw(\phi_U(T))= \svw{13\ 12\ 10\ 8\ 11\  7\,} 6\, \bm{5}\, \bm{4}\,  9\, \bm{3}\, \bm{2}\, \bm{1},
\rw(T)= 8\, 7\, 5\, 3\, 6 \, 2\, 1\, 4,  
\rw(U)=\bm{5}\, \bm{4}\, \bm{3}\, \bm{2}\, \bm{1}.\]
In $\phi_U(T)$,  the 7 entries in $\alpha/\beta$ in rows 3,4,5 each create 5 inversions with the entries in $\beta$, and the entry 9 creates 3 inversions with the entries in $\beta$, so we have 
\[\inv(\alpha, \beta)=7(5)+3=38.\]  
\end{example}

More generally, letting $\ell=\ell(\beta)$ and $\gamma_i=\alpha_i-\beta_i, i=1, \ldots \ell$, one checks that one has the formula
\[\inv(\alpha,\beta)= \gamma_1\cdot \beta_1 +\gamma_2 \cdot (\beta_2+\beta_1) +\cdots+\gamma_\ell\cdot (\beta_\ell+\beta_{\ell-1}+\cdots+\beta_1) 
+(\alpha_{\ell+1}+\cdots+\alpha_{\ell(\alpha)})(\beta_\ell+\beta_{\ell-1}+\cdots+\beta_1)\]
where $\inv(\alpha, \beta)$ is defined to be the number of inversions in the reading word of $\phi(T)$ arising from pairs $(i,j)$ such that $i>m$ and $j\le m$. 

The following lemma is straightforward to verify.

\begin{lemma}\label{lem:skew-rank-phi} Let $\alpha, \beta$ be compositions with $\beta \subseteq \alpha$. Fix $U\in \SIT(\beta)$. Consider the map $\phi_U:\SIT(\alpha/\beta)\rightarrow \SIT(\alpha)$ as described above.
Let $T\in \SIT(\alpha/\beta)$.  Then 
\[\inv (\phi(T))=\inv (T) + \inv (U) +\inv(\alpha, \beta)\]
where $\inv(\alpha, \beta)$ is defined as above. 
\end{lemma}
The poset $P\rdI_{\alpha/\beta}$ is therefore ranked by the number of inversions $\inv(T)=\inv \rw(T)$ for $T\in \SIT(\alpha/\beta)$, and  the next result generalises the formula in \cite[Remark~6.14]{NSvWVW2024} for the rank of the poset $P\rdI_\alpha$.

\begin{proposition}\label{prop:poset-rank-Nadia} Let $\alpha, \beta$ be compositions with $\beta \subseteq \alpha$.
    Then the poset $P\rdI_{\alpha/\beta}$ is graded, and the rank  of the 
    skew standard  immaculate tableau $T$ is given by $\inv (T)-\inv (S_{\alpha/\beta}^0)$. 
    
    In particular,      $\inv(S^{row}_{\alpha/\beta})=\binom{|\alpha|-|\beta|}{2}$, and
    \[\inv (S_{\alpha/\beta}^0)=\sum_{i=1}^{\ell(\beta)} \binom{\alpha_i - \beta_i}{2}+\sum_{i = 
    \ell(\beta)+1}^{\ell(\alpha)}  \binom{\alpha_i+i-\ell(\beta)-1}{2} -
    \binom{\ell(\alpha) - \ell(\beta)}{3}.\]
    Hence the rank  
    of the poset $P\rdI_{\alpha/\beta}$ is \[\binom{|\alpha|-|\beta|}{2}- 
    \sum_{i=1}^{\ell(\beta)} \binom{\alpha_i - \beta_i}{2} - \sum_{i = 
    \ell(\beta)+1}^{\ell(\alpha)}  \binom{\alpha_i+i-\ell(\beta)-1}{2} + 
    \binom{\ell(\alpha) - \ell(\beta)}{3}.\]
\end{proposition} 
\begin{proof}
    Since the poset is ranked by the number of inversions, we need only 
    compute the maximum number of inversions possible in a standard 
    immaculate tableau. We proceed by inclusion-exclusion.\\
    The maximum number of inversions is the number of pairs that can appear 
    in any order. The total number of pairs is $\svw{\binom{|\alpha / \beta|}{2}}$, from which we 
    need to subtract pairs that need to be ordered. To count these pairs, define a \textit{hook} as the collection of the \svw{cells} from $\{(1,1), \ldots, (i-1, 1), (i,1), (i,2), \ldots, (i, \alpha_i)\}$ that appear in the skew diagram $\alpha/\beta$, \svw{where $(i,j)$ refers to the cell or hole in row $i$ and column $j$, and row 1 is the bottom row.} The pairs that need to be ordered are then the ones that 
    appear in the same hook, and note that the hook for row $i$ is just the entries in row $i$ if the cell $(i,1)$ does not belong to $\alpha/\beta$.
    For the hook with the cell $(i,1)$, there are 
    $\binom{\alpha_i+i-\ell(\beta)-1}{2}$ such pairs, since the hook 
    contains $i-\ell(\beta)$ cells in the column and $\alpha_i$ cells in 
    the row, including the cell that belongs to both the row and the 
    column. Similarly, the $\ell(\beta)$ rows that have no cell in the 
    first column have $\sum_{i=1}^{\ell(\beta)}\binom{\alpha_i - 
    \beta_i}{2}$ pairs that need to be ordered, since there are $\alpha_i - 
    \beta_i$ entries in row $i$.
    
    Finally, the pairs of elements that are both in the first column are 
    overcounted by this process, since they may belong to more than one 
    hook. To ensure that those pairs are subtracted only once, we add back 
    the number of pairs strictly below row $i$, for each $\ell(\beta)+1\leq 
    i\leq \ell(\alpha)$:
    \[\sum_{i=\ell(\beta)+1}^{\ell(\alpha)} \binom{i-1-\ell(\beta)}{2} 
    =\sum_{i=1}^{\ell(\alpha) - \ell(\beta)} \binom{i-1}{2} = 
    \binom{\ell(\alpha) - \ell(\beta)}{3}.   \qedhere\]
\end{proof}

 Now we turn our attention back to skew standard extended tableaux.  Recall from \Cref{prop:skew-SET-module} that the subset $\mathrm{span}\{T\in \SIT(\alpha/\beta):T\in\SET(\alpha/\beta\}$  is a 0-Hecke module for both the  $\rdI$-action and   the $\mathcal{A}^*$-action.

\begin{proposition}\label{prop:skew-SET-module-top-element} Let $\alpha ,\beta $ be compositions with $\beta \subseteq \alpha$. Then the subposet of $P\rdI_{\alpha/\beta}$ whose elements are $\SET(\alpha/\beta)$ has a unique maximal element, namely the unique maximal element $S^{row}_{\alpha/\beta}$ of $P\rdI_{\alpha/\beta}$. This  subposet is ranked, with the same rank function as $P\rdI_{\alpha/\beta}$.  For the $\rdI$-action or  the $\mathcal{A}^*$-action, the  minimal elements of the subposet   generate  $\mathrm{span}\{T :T\in\SET(\alpha/\beta\}$ as a \svw{submodule} of  $\mathcal{V}_{\alpha/\beta}$ or  $\mathcal{A}_{\alpha/\beta}$, respectively.
\end{proposition}
\begin{proof}   
Combining  \Cref{prop:skew-SET-module} with the straightening algorithm of  \Cref{prop:top-elt-skew},  shows that 
the subposet whose elements are $\SET(\alpha/\beta)$ has a  unique maximal element $S^{row}_{\alpha/\beta}\in \SET(\alpha/\beta)$. 
\svw{The last two statements follow from \Cref{prop:skew-SET-module}.}  \end{proof}
 
When $\beta\ne \emptyset$, the subposet $\SET(\alpha/\beta)$ 
 need not have a  unique minimal element, as \Cref{fig:SkewSETPoset} shows.  

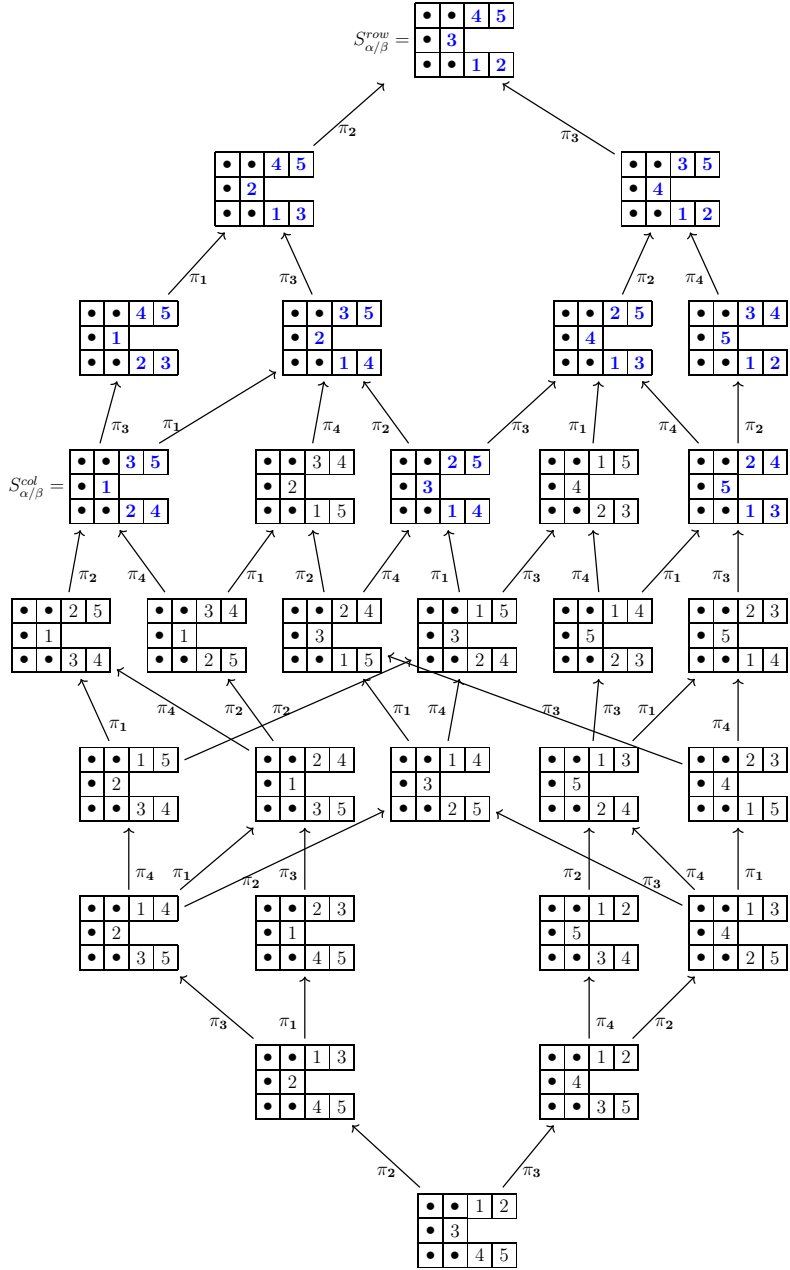
\begin{figure}[htb] \centering
\begin{center}
    \scalebox{.6}{			
        \begin{tikzpicture}
            \newcommand*{\xdist}{*3}
            \newcommand*{\ydist}{*2.2}
            \node (m41) at (0\xdist,-6.5\ydist) 
            {$\tableau{ 
            \bullet &\bullet & 1 & 2 \\ 
            \bullet &  3\\
            \bullet & \bullet &  4 & 5}  $
            } ;
            \node (m31) at (-1.2\xdist,-5\ydist) 
            {$\tableau{ 
            \bullet &\bullet & 1 & 3 \\ 
            \bullet &  2\\
            \bullet & \bullet &  4 & 5}  $
            } ;
            \node (m32) at (0.9\xdist,-5\ydist) 
            {$\tableau{ 
            \bullet &\bullet &1 & 2 \\ 
            \bullet &   4\\
            \bullet & \bullet &  3 & 5}  $
            } ;
            \node (m21) at (-2.5\xdist,-3.5\ydist) 
            {$\tableau{ 
            \bullet &\bullet & 1  & 4 \\ 
            \bullet &  2 \\
            \bullet & \bullet &  3 & 5}$ 
            };
            \node (m22) at (-1.2\xdist,-3.5\ydist) 
            {$\tableau{ 
            \bullet &\bullet & 2 & 3 \\ 
            \bullet &  1\\
            \bullet & \bullet &  4 & 5}  $
            } ;
            \node (m23) at (0.9\xdist,-3.5\ydist) 
            {$\tableau{ 
            \bullet &\bullet &1 & 2 \\ 
            \bullet &   5\\
            \bullet & \bullet &  3 & 4}  $
            } ;
            \node (m24) at (2\xdist,-3.5\ydist) 
            {$\tableau{ 
            \bullet &\bullet & 1 & 3 \\ 
            \bullet &   4\\
            \bullet & \bullet &  2 & 5}  $
            } ;
            \node (m11) at (-2.5\xdist,-2\ydist) 
            {$\tableau{ 
            \bullet &\bullet & 1  & 5 \\ 
            \bullet &  2 \\
            \bullet & \bullet &  3 & 4}$ 
            };
            \node (m12) at (-1.2\xdist,-2\ydist) 
            {$\tableau{ 
            \bullet &\bullet & 2 & 4 \\ 
            \bullet &  1\\
            \bullet & \bullet &  3 & 5}  $
            } ;
            \node (m13) at (-0.2\xdist,-2\ydist) 
            {$\tableau{ 
            \bullet &\bullet &1 & 4 \\ 
            \bullet &   3\\
            \bullet & \bullet &  2 & 5}  $
            } ;
            \node (m14) at (0.9\xdist,-2\ydist) 
            {$\tableau{ 
            \bullet &\bullet &1 & 3 \\ 
            \bullet &   5\\
            \bullet & \bullet &  2 & 4}  $
            } ;
            \node (m15) at (2\xdist,-2\ydist) 
            {$\tableau{ 
            \bullet &\bullet & 2 & 3 \\ 
            \bullet &   4\\
            \bullet & \bullet &  1 & 5}  $
            } ;
            \node (m01) at (-3\xdist,-0.5\ydist) 
            {$\tableau{ 
                    \bullet &\bullet & 2  & 5 \\ 
                    \bullet &  1 \\
                    \bullet & \bullet &  3 & 4}$ 
            };
            \node (m02) at (-2\xdist,-0.5\ydist) 
            {$\tableau{ 
                    \bullet &\bullet & 3 & 4 \\ 
                    \bullet &  1\\
                    \bullet & \bullet &  2 & 5}  $
            } ;
            \node (m03) at (-1\xdist,-0.5\ydist) 
            {$\tableau{ 
                    \bullet &\bullet &2 & 4 \\ 
                    \bullet &   3\\
                    \bullet & \bullet &  1 & 5}  $
            } ;
            \node (m04) at (0\xdist,-0.5\ydist) 
            {$\tableau{ 
                    \bullet &\bullet &1 & 5 \\ 
                    \bullet &   3\\
                    \bullet & \bullet &  2 & 4}  $
            } ;
            \node (m05) at (1\xdist,-0.5\ydist) 
            {$\tableau{ 
                    \bullet &\bullet & 1 & 4 \\ 
                    \bullet &   5\\
                    \bullet & \bullet &  2 & 3}  $
            } ;
            \node (m06) at (2\xdist,-0.5\ydist) 
            {$\tableau{ 
                    \bullet &\bullet &   2&   3\\ 
                    \bullet & 5 \\
                    \bullet & \bullet &  1 & 4}  $
            } ;
            \node (n11) at (-2.8\xdist,1\ydist) 
            {$  S^{col}_{\alpha/\beta}=\tableau{ 
                    \bullet &\bullet & \textbf{\tcblue{3}}  & \textbf{\tcblue{ 5}} \\ 
                    \bullet & \textbf{\tcblue{ 1}} \\
                    \bullet & \bullet & \textbf{\tcblue{ 2}} &\textbf{\tcblue{ 4}} }$
            };
            \node (n112) at (-1.2\xdist,1\ydist) 
            {$\tableau{ 
                    \bullet &\bullet & 3 & 4 \\ 
                    \bullet &  2\\
                    \bullet & \bullet &  1 & 5}  $
            } ;
            \node (n13) at (-0.2\xdist,1\ydist) 
            {$\tableau{ 
                    \bullet &\bullet &\textbf{\tcblue{ 2}} &\textbf{\tcblue{ 5}} \\ 
                    \bullet &  \textbf {\tcblue{3}}\\
                    \bullet & \bullet & \textbf{\tcblue{ 1}} &\textbf{\tcblue{ 4}}}  $
            } ;
            \node (n122) at (0.9\xdist,1\ydist) 
            {$\tableau{ 
                    \bullet &\bullet & 1 & 5 \\ 
                    \bullet &   4\\
                    \bullet & \bullet &  2 & 3}  $
            } ;
            \node (n12) at (2\xdist,1\ydist) 
            {$\tableau{ 
                    \bullet &\bullet &\textbf{\tcblue{   2}}&\textbf {\tcblue{  4}}\\ 
                    \bullet &\textbf {\tcblue{5}} \\
                    \bullet & \bullet & \textbf{\tcblue{ 1}} &\textbf{\tcblue{ 3}}}  $
            } ;
            \node (n31) at (-1\xdist,2.5\ydist) 
            {$\tableau{ 
                    \bullet &\bullet &\textbf{\tcblue{ 3}} &\textbf{\tcblue{   5}} \\ 
                    \bullet &\textbf{\tcblue{ 2}} \\
                    \bullet & \bullet & \textbf{\tcblue{ 1}} &\textbf{\tcblue{ 4}}}  $
            } ;
            \node (n32) at (-2.5\xdist,2.5\ydist) 
            {$\tableau{ 
                    \bullet &\bullet & \textbf {\tcblue{ 4}} &\textbf{\tcblue{   5}} \\ 
                    \bullet &\textbf{\tcblue{   1}}\\
                    \bullet & \bullet & \textbf{\tcblue{ 2}} &\textbf{\tcblue{ 3}}}  $
            } ;
            \node (n33) at (2\xdist,2.5\ydist) 
            {$\tableau{ 
                    \bullet &\bullet &\textbf{\tcblue{  3}} &\textbf{\tcblue{ 4}} \\ 
                    \bullet &\textbf{\tcblue{5}} \\
                    \bullet & \bullet & \textbf{\tcblue{ 1}} &\textbf{\tcblue{ 2}}}  $
            } ;
            \node (n34) at (1\xdist,2.5\ydist) 
            {$\tableau{ 
                    \bullet &\bullet &\textbf{\tcblue{ 2}} &\textbf {\tcblue{ 5}} \\ 
                    \bullet &\textbf{\tcblue{ 4}} \\
                    \bullet & \bullet & \textbf{\tcblue{ 1}} &\textbf{\tcblue{ 3}}}  $
            } ;
            \node (n71) at (-1.5\xdist,4\ydist) 
            {$\tableau{ 
                    \bullet &\bullet & \textbf{\tcblue{  4}} &\textbf{\tcblue{   5}} \\ 
                    \bullet &\textbf{\tcblue{ 2}} \\
                    \bullet & \bullet & \textbf{\tcblue{ 1}} &\textbf{\tcblue{ 3}}}  $
            } ;
            \node (n72) at (1.5\xdist,4\ydist) 
            {$\tableau{ 
                    \bullet &\bullet &\textbf{\tcblue{ 3}} &\textbf{\tcblue{   5}} \\ 
                    \bullet &\textbf {\tcblue{4}} \\
                    \bullet & \bullet & \textbf{\tcblue{ 1}} &\textbf{\tcblue{ 2}}}  $
            } ;
            \node (n8) at (-0.25\xdist,5.5\ydist) 
            {  ${S^{row}_{\alpha/\beta} =
                    \tableau{ 
                        \bullet &\bullet & \textbf{\tcblue{  4}} &\textbf{\tcblue{   5}} \\ 
                        \bullet &\textbf{\tcblue{3}} \\
                        \bullet & \bullet & \textbf{\tcblue{ 1}} &\textbf{\tcblue{ 2}}} }  $ 
            }; 
            
            
            \draw [thick, ->] (n71) -- (n8) node [near start, right] 
            {$\bf \pi_{2}$}; 
            \draw [thick, ->] (n72) -- (n8) node [near start, left] 
            {$\bf \pi_{3}$};

            \draw [thick, ->] (n11) -- (n31) node [near start, left] 
            {$\bf \pi_{1}$}; 
            \draw [thick, ->] (n11) -- (n32) node [near start, right] 
            {$\bf \pi_{3}$};
            
            \draw [thick, ->] (n12) -- (n33) node [near start, right] 
            {$\bf \pi_{2}$}; 
            \draw [thick, ->] (n12) -- (n34) node [near start, left] 
            {$\bf \pi_{4}$}; 
            \draw [thick, ->] (n13) -- (n31) node [near start, left] 
            {$\bf \pi_{2}$}; 
            \draw [thick, ->] (n13) -- (n34) node [near start, right] 
            {$\bf \pi_{3}$}; 
            
            \draw [thick, ->] (n31) -- (n71) node [near start, left] 
            {$\bf \pi_{3}$}; 
            \draw [thick, ->] (n32) -- (n71) node [near start, right] 
            {$\bf \pi_{1}$}; 
            
            \draw [thick, ->] (n33) -- (n72) node [near start, left] 
            {$\bf \pi_{4}$}; 
            \draw [thick, ->] (n34) -- (n72) node [near start, right] 
            {$\bf \pi_{2}$};

            \draw [thick, ->] (n112) -- (n31) node [near start, right] 
            {$\bf \pi_{4}$}; 
            \draw [thick, ->] (n122) -- (n34) node [near start, left] 
            {$\bf \pi_{1}$};

             \draw [thick, ->] (m01) -- (n11) node [near start, right] 
            {$\bf \pi_{2}$}; 
            \draw [thick, ->] (m02) -- (n11) node [near start, left] 
            {$\bf \pi_{4}$};      
             \draw [thick, ->] (m02) -- (n112) node [near start, right] 
            {$\bf \pi_{1}$}; 
            \draw [thick, ->] (m03) -- (n112) node [near start, left] 
            {$\bf \pi_{2}$};      
             \draw [thick, ->] (m03) -- (n13) node [near start, right] 
            {$\bf \pi_{4}$}; 
            \draw [thick, ->] (m04) -- (n13) node [near start, left] 
            {$\bf \pi_{1}$};      
             \draw [thick, ->] (m04) -- (n122) node [near start, right] 
            {$\bf \pi_{3}$}; 
            \draw [thick, ->] (m05) -- (n122) node [near start, left] 
            {$\bf \pi_{4}$};      
             \draw [thick, ->] (m05) -- (n12) node [near start, right] 
            {$\bf \pi_{1}$}; 
            \draw [thick, ->] (m06) -- (n12) node [near start, left] 
            {$\bf \pi_{3}$};      
            
             \draw [thick, ->] (m11) -- (m01) node [near start, right] 
            {$\bf \pi_{1}$}; 
            \draw [thick, ->] (m11) -- (m04) node [midway, 
            left] {$\bf \pi_{2}$};      
            \draw [thick, ->] (m12) -- (m01) node [midway, left] 
            {$\bf \pi_{4}$}; 
            \draw [thick, ->] (m12) -- (m02) node [midway, left] 
            {$\bf \pi_{2}$};      
            \draw [thick, ->] (m13) -- (m03) node [midway, right] 
            {$\bf \pi_{1}$}; 
            \draw [thick, ->] (m13) -- (m04) node [midway, left] 
            {$\bf \pi_{4}$};      
            \draw [thick, ->] (m14) -- (m05) node [midway, right] 
            {$\bf \pi_{3}$}; 
            \draw [thick, ->] (m14) -- (m06) node [midway, left] 
            {$\bf \pi_{1}$};      
            \draw [thick, ->] (m15) -- (m03) node [midway,
            right] {$\bf \pi_{3}$}; 
            \draw [thick, ->] (m15) -- (m06) node [near start, left] 
            {$\bf \pi_{4}$};

             \draw [thick, ->] (m21) -- (m11) node [near start, right] 
            {$\bf \pi_{4}$}; 
            \draw [thick, ->] (m21) -- (m12) node [near start, left] 
            {$\bf \pi_{1}$};      
            \draw [thick, ->] (m21) -- (m13) node [near start, right] 
            {$\bf \pi_{2}$}; 
            \draw [thick, ->] (m22) -- (m12) node [near start, left] 
            {$\bf \pi_{3}$};      
            \draw [thick, ->] (m23) -- (m14) node [near start, left] 
            {$\bf \pi_{2}$}; 
            \draw [thick, ->] (m24) -- (m13) node [near start, right] 
            {$\bf \pi_{3}$};  
            \draw [thick, ->] (m24) -- (m14) node [near start, right] 
            {$\bf \pi_{4}$};     
            \draw [thick, ->] (m24) -- (m15) node [near start, right] 
            {$\bf \pi_{1}$}; 
             
             \draw [thick, ->] (m31) -- (m21) node [near start, left] 
             {$\bf \pi_{3}$};      
             \draw [thick, ->] (m31) -- (m22) node [near start, left] 
             {$\bf \pi_{1}$}; 
             \draw [thick, ->] (m32) -- (m23) node [near start, right] 
             {$\bf \pi_{4}$};      
             \draw [thick, ->] (m32) -- (m24) node [near start, right] 
             {$\bf \pi_{2}$}; 
             
             \draw [thick, ->] (m41) -- (m31) node [near start, left] 
             {$\bf \pi_{2}$};      
             \draw [thick, ->] (m41) -- (m32) node [near start, right] 
             {$\bf \pi_{3}$};

        \end{tikzpicture}	
   }			
\end{center}
\caption{\svw{The poset $P\rdI_{\alpha/\beta}$} for $\alpha=(4,2,4), \beta=(2,1,2)$.  The subposet \svw{whose elements are} $\SET(\alpha/\beta)$,  with three minimal elements, is  \svw{in bold} blue.}\label{fig:SkewSETPoset}
\end{figure}

\begin{example}\label{ex:Scol-not-minimal-for-SET1} Let $\alpha=(2,3,2)$, $\beta=(1,2,1)$. Then 

\[\SET(\alpha/\beta)=\left\{S^{col}_{\alpha/\beta}=\tableau{\bullet  & 2\\ \bullet & \bullet &3\\ \bullet &1},\quad T_1=\tableau{\bullet  &3\\ \bullet &\bullet &1\\ \bullet &2}, \quad T_2= S^{row}_{\alpha/\beta}=\tableau{\bullet  &3\\ \bullet &\bullet &2\\ \bullet &1}\right \}.\]
Note that $\pi_1(T_1)=T_2=\pi_2(S^{col}_{\alpha/\beta})$, but $T_1$ and $S^{col}_{\alpha/\beta}$ are incomparable and are thus minimal elements of 
$\SET(\alpha/\beta)$.

\end{example}

The next example shows that, although $S^{row}_{\alpha/\beta}$ is always the maximal element, 
 $S^{col}_{\alpha/\beta}$ need not be a minimal element of $\SET(\alpha/\beta)$.

\begin{example}\label{ex:Scol-not-minimal-for-SET2} Let $\alpha=(2,3)$, $\beta=(1,2)$. Then 
\[\SIT(\alpha/\beta)=\SET(\alpha/\beta)=\left\lbrace T=\tableau{\bullet &\bullet & 1\\ \bullet &2}, \quad S^{col}_{\alpha/\beta}=S^{row}_{\alpha/\beta}=\tableau{\bullet &\bullet &2\\ \bullet &1}\right\rbrace. \]

Here $\pi_1(T)=S^{col}_{\alpha/\beta}=S^{row}_{\alpha/\beta}$ and so $T$ is the minimal element. 

\end{example}

In analogy with \cite[Definition 7.14]{NSvWVW2024}, we make the following definition. 

\begin{definition}\label{def:skew-NSET} Let $\alpha, \beta$ be compositions with $\beta \subseteq \alpha$ and $\mathrm{NSET}(\alpha/\beta)$ be the subset of all skew standard  tableaux  of shape $\alpha/\beta$ in which all  rows increase left to right, but at least one {column} does NOT increase from bottom to top. 
\end{definition}

We now observe that, as in \cite[Proposition 7.25]{NSvWVW2024},  the poset \svw{$P\rdI _{\alpha/\beta}$} also gives a quotient module of  $\mathcal{W}_{\alpha/\beta}$ with characteristic equal to the extended Schur function $\mathcal{E}_{\alpha/\beta}$. 
In  Figure~\ref{fig:SkewSETPoset}, by reversing the arrows, one sees that  the tableaux that are NOT in $\SET(\alpha/\beta)$ form a 
  closed subset under the $\dI$-action.

\begin{proposition}\label{prop:skew-shin-module} Let $\alpha, \beta$ be compositions with $\beta \subseteq \alpha$. Then $\NSET(\alpha/\beta)\cap\SIT(\alpha/\beta)$ is a basis for a  \svw{submodule} $\mathcal{Y}_{\alpha/\beta}$ of $\mathcal{W}_{\alpha/\beta}$ for the $\dI$-action.   The resulting quotient module 
$\mathcal{W}_{\alpha/\beta}/\mathcal{Y}_{\alpha/\beta}$ has basis of cosets represented by the set $\SET(\alpha/\beta)$,  and is cyclically generated by (the coset represented by) $S^{row}_{\alpha/\beta}$.  A similar statement holds for the $\bA^*$-action. Their quasisymmetric characteristics are respectively $\mathcal{E}_{\alpha/\beta}$ and $\bA ^* _{\alpha/\beta}$.
\end{proposition}

\begin{proof} First  recall  from \Cref{lem:sameposet-skew} that  the $\dI$-action is \svw{obtained from} the $\rdI$-action in the 0-Hecke poset \svw{by reversing the arrows.}
Proposition~\ref{prop:skew-SET-module} tells us that  the subspace spanned by $\NSET(\alpha/\beta)$ is closed under the $\dI$-action, since 
$T=\pi _i ^{\dI}(S) \iff S=\pi_i^{\rdI}(T) $; hence if $S\notin \SET(\alpha/\beta)$, then \svw{$T\notin \SET(\alpha/\beta)$.} 
Now the key observation  is that by \Cref{prop:top-elt-skew}, the subposet of $P\rdI_{\alpha/\beta}$ consisting of $\SET(\alpha/\beta)$ has the unique maximal element $S^{row}_{\alpha/\beta }$ by its definition, which is thus the unique minimal element for the dual poset $P\dI_{\alpha/\beta }$.  This is therefore the cyclic generator of the quotient module for the $\dI$-action. 
The proof then follows as in  \cite[Proposition~7.25]{NSvWVW2024}.
\end{proof}

We conclude with the analogue of Theorem~\ref{thm:Skew-branch-all4flavours} for $\SET(\alpha)$.

\begin{theorem}\label{thm:Skew-branch-SET-all4flavours} Let $\alpha\vDash n$, and  let $\mathcal{Z}_\alpha$ denote the $\hn$-module defined by any one of the four actions  on the  set $\SET(\alpha)$ arising from  the four descent sets $\Des_{\dI}$, $\Des_{\rdI}$, $\Des_{\mathcal{A}^*}$, $\Des_{\bA^*}$, where for $\Des_{\dI}$, $\Des_{\bA^*}$ they are quotient actions.
 Then for fixed $m\leq n$ we have 
\[\mathcal{Z}_{\alpha}\big\downarrow^{\hn}_{H_m(0)\,\otimes\, H_{n-m}(0)} = 
\bigoplus_{\substack{\beta\vDash m  \\ \beta\subseteq\alpha}} \,
\mathcal{Z}_\beta \otimes \mathcal{Z}_{\alpha/\beta}.
\quad\text{ In particular, }\quad
\mathcal{Z}_\alpha\big\downarrow^{\hn}_{H_{n-1}(0)} \cong \bigoplus_{\substack{\beta\vDash n-1  \\ \beta\subset\alpha}} \mathcal{Z}_\beta . \]
\end{theorem}

\noindent
\textbf{Acknowledgments.} This project was initiated in January 2024 at  the   \textit{Community in Algebraic and Enumerative Combinatorics} workshop, held at the Banff International Research Station (BIRS).   The authors gratefully acknowledge the  support of BIRS, as well as partial support from the Natural Sciences and Engineering Research Council of Canada. \svw{This work was partially supported by NSF grant DMS–2153998.}

\bibliographystyle{plain}
\bibliography{BIRS2024PAPER-FINAL}
\end{document}